\documentclass{amsart}

\usepackage{pstricks,pst-node,graphicx}
\usepackage{amsmath,amssymb,stmaryrd}

\newtheorem{theorem}{Theorem}[section]
\newtheorem{lemma}[theorem]{Lemma}
\newtheorem{proposition}[theorem]{Proposition}
\newtheorem{corollary}[theorem]{Corollary}

\theoremstyle{definition}
\newtheorem{definition}[theorem]{Definition}

\theoremstyle{remark}

\numberwithin{equation}{section}

\newcommand{\SSYT}{\ensuremath\mathrm{SSYT}}
\newcommand{\SYT}{\ensuremath\mathrm{SYT}}
\newcommand{\shSSYT}{\ensuremath\mathrm{SSShT}}

\newcommand{\shSYT}{\ensuremath\mathrm{SShT}}
\newcommand{\siSYT}{\ensuremath\mathrm{SShT}_{\pm}}
\newcommand{\siSYTp}{\ensuremath\mathrm{SShT}_{\pm}^{\prime}}

\newcommand{\A}{\ensuremath\mathcal{A}}
\newcommand{\C}{\ensuremath\mathcal{C}}
\newcommand{\D}{\ensuremath\mathcal{D}}

\newcommand{\Des}{\ensuremath\mathrm{Des}}
\newcommand{\Peak}{\ensuremath\mathrm{Peak}}
\newcommand{\Spike}{\ensuremath\mathrm{Spike}}

\newcommand{\sd}{\ensuremath\mathrm{b}}

\newlength\cellsize 
\savebox2{%
\begin{picture}(12,12)
\put(0,0){\line(1,0){12}}
\put(0,0){\line(0,1){12}}
\put(12,0){\line(0,1){12}}
\put(0,12){\line(1,0){12}}
\end{picture}}

\newcommand\cellify[1]{\def\thearg{#1}\def\nothing{}%
\ifx\thearg\nothing
\vrule width0pt height\cellsize depth0pt\else
\hbox to 0pt{\usebox2\hss}\fi%
\vbox to 12\unitlength{
\vss
\hbox to 12\unitlength{\hss$#1$\hss}
\vss}}

\newcommand\tableau[1]{\vtop{\let\\=\cr
\setlength\baselineskip{-12000pt}
\setlength\lineskiplimit{12000pt}
\setlength\lineskip{0pt}
\halign{&\cellify{##}\cr#1\crcr}}}

\newcommand{\cs}[1]{c@{\hskip #1ex}}
\newcommand{\B}{\bullet}

\newcommand{\e}{\mbox{}}
\definecolor{boxgray}{gray}{.8}
\newcommand{\cb}{\color{boxgray}\rule{1\cellsize}{1\cellsize}\hspace{-\cellsize}\usebox2}

\begin{document}

%%%%%%%%%%%%%%%%%%%%%%%%%%%%%%%%%%%%%%%%%%%%%%%%%%%%%%%%%%%%
%  TITLE PAGE information
%%%%%%%%%%%%%%%%%%%%%%%%%%%%%%%%%%%%%%%%%%%%%%%%%%%%%%%%%%%%

%     [Short Title]{Full Title}
\title[Shifted dual equivalence]{Shifted dual equivalence and Schur $P$-positivity}

\author[S. Assaf]{Sami Assaf}
\address{Department of Mathematics, University of Southern California, Los Angeles, CA 90089-2532}  
\email{shassaf@usc.edu}
\thanks{Work supported in part by NSF grant DMS-1265728.}

%    General info

%\dedicatory{}


%\date{\today}

\begin{abstract}
  By considering type B analogs of permutations and tableaux, we
  extend abstract dual equivalence to type B in two directions. In one
  direction, we define involutions on shifted tableaux that give a
  dual equivalence, thereby giving a new combinatorial proof of the
  Schur positivity of Schur $Q$- and $P$-functions. In another
  direction, we define an abstract shifted dual equivalence parallel
  to dual equivalence and prove that it can be used to establish Schur
  $P$-positivity of a function expressed as a sum of shifted
  fundamental quasisymmetric functions. As a first application, we
  give a new combinatorial proof that the product of Schur
  $P$-functions is Schur $P$-positive.
\end{abstract}

\maketitle
%\tableofcontents

%%%%%%%%%%%%%%%%%%%%%%%%%%%%%%%%%%%%%%%%%%%%%%%%%%%%%%%%%%%%%%%%
\section{Introduction}
%%%%%%%%%%%%%%%%%%%%%%%%%%%%%%%%%%%%%%%%%%%%%%%%%%%%%%%%%%%%%%%%
\label{s:introduction}

Symmetric function theory can be harnessed by other areas of
mathematics to answer fundamental enumerative questions. For example,
multiplicities of irreducible components, dimensions of algebraic
varieties, and various other algebraic constructions that require the
computation of certain integers may often be translated to the
computation of the coefficients of a given function in a particular
basis. Often the chosen basis is the Schur functions, which arise as
Frobenius characters of irreducible representations of the symmetric
group and as Schubert polynomials for the cohomology ring of the
Grassmannian. Thus a quintessential problem in symmetric functions is
to prove that a given function has nonnegative integer coefficients
when expressed as a sum of Schur functions.

In \cite{Assaf-DEGs-I}, the author introduced dual equivalence graphs
as a universal tool by which one can approach such problems. This tool
has been applied to various important classes of symmetric functions,
include LLT and Macdonald polynomials \cite{Assaf-DEGs-II}, $k$-Schur
functions \cite{AssafBilley2012}, and products of Schubert polynomials
\cite{ABS2014}.

In this paper, we give a further application of dual equivalence to
Schur $Q$- and $P$-functions \cite{Schur1911}. These functions arise
in the study of projective representation of the symmetric group
\cite{Stembridge1989} as well as the cohomology classes dual to
Schubert cycles in isotropic Grassmannians
\cite{Jozefiak1991,Pragacz1991}. These functions enjoy many nice
properties parallel to Schur functions \cite{Macdonald}. In
particular, they form dual bases for an important subspace of
symmetric functions. While they have long been known to be Schur
positive \cite{Sagan1987} and to have positive structure constants
\cite{Stembridge1989}, the new proofs we provide lay the foundation
for a stronger extension of dual equivalence to type B. We define an
abstract notion of shifted dual equivalence that offers a tool by
which one can show that a given function has nonnegative coefficients
when expanded in terms of Schur $P$-functions. As a first application,
we consider the Schur $P$-expansion of a product Schur
$P$-functions. Related work by \cite{BilleyEtAl} holds
further applications.

This paper is organized as follows. In Section~\ref{s:tableaux}, we
introduce the classic combinatorial objects and their type B
analogs. We connect the combinatorics with symmetric and
quasisymmetric functions in Sections~\ref{s:symmetric} and
\ref{s:quasi}. In Section~\ref{s:dual_equivalence}, we review abstract
dual equivalence, and we give an application to type B combinatorial
objects in Section~\ref{s:signed} proving that Schur $P$-functions are
Schur positive. In Section~\ref{s:shifted}, we generalize the
definitions and theorems of dual equivalence to the type B setting and
define an abstract notion of shifted dual equivalence. Our main
result, Theorem~\ref{thm:positivity_sh}, is that this provides a
universal tool for establishing Schur $P$-positivity. In
Section~\ref{s:weak}, we apply this new theory to give a new proof
that the product of Schur $P$-functions is Schur $P$-positive.

%%%%%%%%%%%%%%%%%%%%%%%%%%%%%%%%%%%%%%%%%%%%%%%%%%%%%%%%%%%%%%%%
\section{Partitions and tableaux}
%%%%%%%%%%%%%%%%%%%%%%%%%%%%%%%%%%%%%%%%%%%%%%%%%%%%%%%%%%%%%%%%
\label{s:tableaux}

The main combinatorial objects we study are partitions, tableaux, and
permutations, with their type B analogs being strict partitions,
shifted tableaux, and signed permutations.

A \emph{partition} $\lambda$ is a non-increasing sequence of positive
integers, $\lambda = (\lambda_1,\lambda_2, \ldots, \lambda_{\ell})$,
where $\lambda_1 \geq \lambda_2 \geq \cdots \geq \lambda_{\ell} > 0$.
A \emph{strict partition} $\gamma$ is a partition whose parts are
strictly decreasing, i.e. $\gamma_1 > \gamma_2 > \cdots >
\gamma_{\ell} > 0$. The \emph{size} of a partition is the sum of its
parts, i.e. $\lambda_1 + \lambda_2 + \cdots + \lambda_{\ell}$.

We identify a partition $\lambda$ with its Young diagram, the
collection of left-justified cells with $\lambda_i$ cells in row
$i$. For a strict partition $\gamma$, the shifted Young diagram is the
Young diagram with row $i$ shifted $\ell(\gamma)-i$ cells to the
left. It will often be useful to consider the shifted symmetric
diagram for $\gamma$, which is obtained by adjoining the reflection of
the shifted diagram. For examples, see Figure~\ref{fig:shapes}.

\begin{figure}[ht]
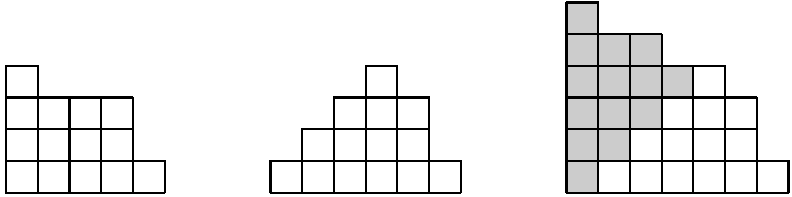

  \begin{center}
    \begin{displaymath}
      \tableau{\\ \\ \e \\ \e & \e & \e & \e \\ \e & \e & \e & \e \\ \e & \e & \e & \e & \e}      
      \hspace{4em}
      \tableau{\\ \\  & & & \e \\ &  & \e & \e & \e \\ & \e & \e & \e & \e \\ \e & \e & \e & \e & \e & \e}
      \hspace{4em}
      \tableau{ \cb \\ \cb & \cb & \cb \\ \cb & \cb & \cb & \cb & \e \\ \cb & \cb & \cb & \e & \e & \e \\ \cb & \cb & \e & \e & \e & \e \\ \cb & \e & \e & \e & \e & \e & \e}
    \end{displaymath}
    \caption{\label{fig:shapes}The Young diagram for $(5,4,4,1)$, and the
      shifted Young diagram and shifted symmetric diagram for $(6,4,3,1)$.}
  \end{center}
\end{figure}

A \emph{semi-standard Young tableau} of shape $\lambda$ is a filling
of the Young diagram for $\lambda$ with positive integers such that
entries weakly increase along rows and strictly increase up
columns. For example, see Figure~\ref{fig:SSYT}.

\begin{figure}[ht]
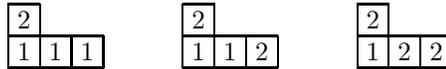

  \begin{center}
    \begin{displaymath}
      \tableau{2 \\ 1 & 1 & 1} \hspace{3em}
      \tableau{2 \\ 1 & 1 & 2} \hspace{3em}
      \tableau{2 \\ 1 & 2 & 2}    
    \end{displaymath}
    \caption{\label{fig:SSYT}The semi-standard Young tableaux of shape
      $(3,1)$ with entries in $\{1,2\}$.}
  \end{center}
\end{figure}

A \emph{semi-standard shifted tableau} of shape $\gamma$ is a filling
of the shifted Young diagram for $\gamma$ with marked or unmarked
positive integers such that entries weakly increase along rows and
columns according to the ordering $1^{\prime} < 1 < 2^{\prime} < 2 <
\cdots$, each row has at most one marked entry $i^{\prime}$ for each
$i$ and each column has at most one unmarked entry $i$ for each
$i$. For examples, see Figure~\ref{fig:shSSYT}. 

\begin{figure}[ht]
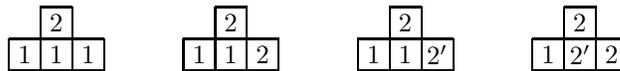

  \begin{center}
    \begin{displaymath}
      \tableau{& 2 \\ 1 & 1 & 1} \hspace{3em}
      \tableau{& 2 \\ 1 & 1 & 2} \hspace{3em}
      \tableau{& 2 \\ 1 & 1 & 2^{\prime}} \hspace{3em}
      \tableau{& 2 \\ 1 & 2^{\prime} & 2}    
    \end{displaymath}
    \caption{\label{fig:shSSYT}The semi-standard shifted tableaux of
      shape $(3,1)$ with entries in $\{1^{\prime},1,2^{\prime},2\}$
      and no marked entries on the main diagonal.}
  \end{center}
\end{figure}

Note that these latter conditions allow any entry along the main
diagonal to be marked or unmarked without changing the allowed entries for other positions. If one considers instead the
shifted symmetric diagram and leaves unmarked letters in place while
reflecting the marked letters, then the conditions for a semi-standard
tableau are the same for straight and for shifted shapes. For example,
see Figure~\ref{fig:symSSYT}.

\begin{figure}[ht]
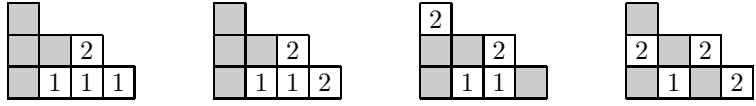

  \begin{center}
    \begin{displaymath}
      \tableau{\cb \\ \cb & \cb & 2 \\ \cb & 1 & 1 & 1} \hspace{3em}
      \tableau{\cb \\ \cb & \cb & 2 \\ \cb & 1 & 1 & 2} \hspace{3em}
      \tableau{2 \\ \cb & \cb & 2 \\ \cb & 1 & 1 & \cb} \hspace{3em}
      \tableau{\cb \\ 2 & \cb & 2 \\ \cb & 1 & \cb & 2}    
    \end{displaymath}
    \caption{\label{fig:symSSYT}The semi-standard shifted symmetric
      tableaux of shape $(3,1)$ with entries in $\{1,2\}$ and no
      reflected entries on the main diagonal.}
  \end{center}
\end{figure}

The \emph{reading word} of a semi-standard tableau $T$, denoted
$w(T)$, is the word obtained by reading the rows of $T$ left to right,
from top to bottom. For example, the reading words for the tableaux in
Figure~\ref{fig:SSYT} from left to right are $2111, 2112, 2122$, and
the reading words for the shifted symmetric tableaux in Figure~\ref{fig:symSSYT}
from left to right are $2111, 2112, 2211, 2212$.

A \emph{permutation} of $n$ is an ordering of the numbers
$\{1,2,\ldots,n\}$. A semi-standard tableau $T$ is \emph{standard} if its
reading word is a permutation. For example, see Figures~\ref{fig:SYT}
and \ref{fig:shSYT}. 

\begin{figure}[ht]
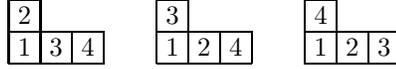

  \begin{center}
    \begin{displaymath}
      \tableau{2 \\ 1 & 3 & 4} \hspace{2em} \tableau{3 \\ 1 & 2 &
        4} \hspace{2em} \tableau{4 \\ 1 & 2 & 3}
    \end{displaymath}
    \caption{\label{fig:SYT}The standard Young tableaux of shape
      $(3,1)$.} 
  \end{center}
\end{figure}

A semi-standard shifted symmetric tableau is a \emph{standard marked tableau} if
its reading word is a permutation and it has entries in the reflected side. For example, each of the two
standard shifted tableaux in Figure~\ref{fig:shSYT} has $2^4$ marked
analogs of which $2^2$ have no signs on the main diagonal; see
Figure~\ref{fig:dual_shifted}.

\begin{figure}[ht]
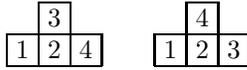

  \begin{center}
    \begin{displaymath}
      \tableau{& 3 \\ 1 & 2 & 4} \hspace{2em}
      \tableau{& 4 \\ 1 & 2 & 3}    
    \end{displaymath}
    \caption{\label{fig:shSYT}The standard shifted tableaux of shape
      $(3,1)$.} 
  \end{center}
\end{figure}

The \emph{descent set} of a permutation is given by
\begin{equation}
  \Des(w) \ = \ \left\{ i \ | \ \mbox{$i$ right of $i+1$} \right\}.
  \label{e:des}
\end{equation}
When $w$ is a permutation of length $n$, we have $\Des(w) \subseteq
\{1,2,\ldots,n-1\}$. When we wish to emphasize $n$, we write
$\Des_n$. Note that there are $2^{n-1}$ possible descent sets for
permutations of length $n$. The \emph{descent set} of a standard
tableau or a marked standard tableau is the descent set of its reading
word. For the tableaux in Figure~\ref{fig:SYT}, the descent sets from
left to right are $\{1\}, \{2\}, \{3\}$, and for the tableaux in
Figure~\ref{fig:shSYT}, the descent sets from left to right are
$\{2\}, \{3\}$. 

In addition to the descent set, we will often be interested in the
\emph{peak set} and the \emph{spike set}, which can be derived
directly from the descent set. For a set $D$, we have
\begin{eqnarray}
  \Spike(D) & = & \{ i \ | \ \mbox{$i-1\not\in D$ and $i\in D$ or
    $i-1\in D$ and $i\not\in D$}\},
  \label{e:spikes} \\
  \Peak(D) & = & \{ i \ | \ \mbox{$i-1\not\in D$ and $i\in D$}\}.
  \label{e:peaks}
\end{eqnarray}
Note that if $D \subseteq \{1,2,\ldots,n-1\}$, then $\Spike(D),
\Peak(D) \subseteq \{2,3,\ldots,n-1\}$. Furthermore, peak sets are
characterized as subsets containing no consecutive entries. Thus there
are $2^{n-2}$ possible spike sets and $F_n$, the $n$th Fibonacci
number, possible peak sets for permutations of length $n$. As with
descents, when we wish to emphasize $n$, we write $\Spike_n$ or
$\Peak_n$.

%%%%%%%%%%%%%%%%%%%%%%%%%%%%%%%%%%%%%%%%%%%%%%%%%%%%%%%%%%%%%%%%
\section{Symmetric functions}
%%%%%%%%%%%%%%%%%%%%%%%%%%%%%%%%%%%%%%%%%%%%%%%%%%%%%%%%%%%%%%%%
\label{s:symmetric}

We follow notation from \cite{Macdonald} for the classic bases for
$\Lambda$, the ring of symmetric functions. The space $\Lambda_n$ of
symmetric functions homogeneous of degree $n$ has dimension equal to
the number of partitions of $n$, and so bases for $\Lambda_n$ are
naturally indexed by partitions of $n$. The most fundamental basis for
$\Lambda_n$ is the Schur function basis, which may be defined by
\begin{equation}
  s_{\lambda}(X) \ = \ \sum_{T \in \SSYT(\lambda)} X^T,
  \label{e:schur}
\end{equation}
where $\SSYT(\lambda)$ denotes the set of all semi-standard Young
tableaux of shape $\lambda$, and $X^T$ is the monomial where $x_i$
occurs in $X^T$ with the same multiplicity with which $i$ occurs in
$T$. For example, the three tableaux in Figure~\ref{fig:SSYT}
contribute $x_1^3 x_2 + x_1^2 x_2^2 + x_1 x_2^3$ to the Schur function
$s_{(3,1)}(X)$. 

The irreducible characters of the symmetric group map under the
Frobenius isomorphism to Schur functions. Therefore Schur functions
are fundamental to understanding representations of the symmetric
group.

Schur $Q$-functions, indexed by strict partitions, are given by
\begin{equation}
  Q_{\gamma} (X) \ = \ \sum_{S \in \shSSYT(\gamma)} X^{|S|},
\label{e:schur_Q}
\end{equation}
where $\shSSYT(\gamma)$ denotes the set of all semi-standard shifted
tableaux of shifted shape $\gamma$, and $X^{|S|}$ is the monomial
where $x_i$ occurs in $X^{|S|}$ with the same multiplicity with which
$i$ and $i^{\prime}$ occur in $S$. For example, the four tableaux in
Figure~\ref{fig:shSSYT} contribute $x_1^3 x_2 + 2 x_1^2 x_2^2 + x_1
x_2^3$ to the Schur $Q$-function $Q_{(3,1)}(X)$.

Schur $P$-functions indexed by strict partitions are given by
\begin{equation}
  P_{\gamma} (X) \ = \ 2^{-\ell(\gamma)} Q_{\gamma}(X) \ =
  \ \sum_{S \in \shSSYT^{*}(\gamma)} X^{|S|},
\label{e:schur_P}
\end{equation}
where $\shSSYT^{*}(\gamma)$ denotes the set of all semi-standard
shifted tableaux of shifted shape $\gamma$ where the main diagonal
has no marked entries, and $X^{|S|}$ is again the monomial where $x_i$
occurs in $X^{|S|}$ with the multiplicity with which $i$ and
$i^{\prime}$ occur in $S$. The second equality follows easily from the
first if one notes that the rules for which entries may be marked
never precludes a marked entry along the main diagonal. 

Schur $P$-functions are fundamental to understanding projective
representations of the symmetric group \cite{Stembridge1989} similar
to the role of Schur functions for linear representations.

Let $\Gamma \subset \Lambda$ denote the subspace of symmetric
functions generated by the odd power sum symmetric functions. The
graded component $\Gamma_n = \Gamma \cap \Lambda_n$ has dimension
equal to the number of strict partitions of $n$. The Schur $Q$- and
$P$-functions may be realized as specializations of Hall-Littlewood
functions at $t=-1$ \cite{Macdonald}. For $\lambda$ a nonstrict
partition, the specialization $Q_{\lambda}(X;-1)$ vanishes, but for
$\lambda$ strict these specializations form dual bases for $\Gamma_n$.

Since the Schur $Q$- and $P$-functions are symmetric, they can be
expanded in the Schur basis. Since the Schur $P$-functions form a
basis for $\Gamma$, the product of two Schur $P$-functions may be
expanded in the Schur $P$-function basis. The machinery we develop in
this paper reproves the following positivity results.

\begin{theorem}
  For $\gamma,\delta$ strict partitions, if
    \begin{equation}
      P_{\gamma} (X) \ = \ \sum_{\lambda} g_{\gamma,\lambda}
      s_{\lambda}(X)
      \hspace{2em} \mbox{and} \hspace{2em}
      P_{\gamma} (X) P_{\delta} (X) \ = \ \sum_{\varepsilon}
      f_{\gamma,\delta}^{\varepsilon} P_{\varepsilon}(X),
      \label{e:P_schur}
    \end{equation}
    then $g_{\gamma,\lambda}, f_{\gamma,\delta}^{\varepsilon} $ are
    nonnegative integers.
\label{thm:P_pos}
\end{theorem}

Stanley conjectured the positivity of $g_{\gamma,\lambda}$, which
follows as a corollary to Sagan's shifted insertion \cite{Sagan1987}
independently developed by Worley \cite{Worley1984}. These ideas were
extended by Stembridge \cite{Stembridge1989} in his study of
projective representations of the symmetric group to give a proof of
the positivity of $f_{\gamma,\delta}^{\varepsilon}$. More recently, Cho
\nocite{Cho2013} built on work of Serrano \cite{Serrano2010} to give
another positivity proof.

One of the main results of this paper is to give another combinatorial
proof of Theorem~\ref{thm:P_pos} using dual equivalence and shifted
dual equivalence, respectively.

%%%%%%%%%%%%%%%%%%%%%%%%%%%%%%%%%%%%%%%%%%%%%%%%%%%%%%%%%%%%%%%%
\section{Quasisymmetric functions}
%%%%%%%%%%%%%%%%%%%%%%%%%%%%%%%%%%%%%%%%%%%%%%%%%%%%%%%%%%%%%%%%
\label{s:quasi}

The space of quasisymmetric functions contains $\Lambda$ and provides
nice intermediate bases for \eqref{e:schur} and for \eqref{e:schur_Q},
\eqref{e:schur_P}. The subspace of quasisymmetric functions
homogeneous of degree $n$ has dimension $2^{n-1}$ and, as such, is
naturally indexed by subsets of $\{1,2,\ldots,n-1\}$.  Gessel's
fundamental basis for quasisymmetric functions \cite{Gessel1984} is
given by
\begin{equation}
  F_{D}(X) \ = \
  \sum_{\substack{i_1 \leq \cdots \leq i_n \\ j \in D \Rightarrow i_j < i_{j+1}}} 
  x_{i_1} \cdots x_{i_n} .
  \label{e:quasisym}
\end{equation}
Implicit in our notation is that $D$ is regarded as a descent set for
objects of size $n$. When we wish to make this explicit, we write
$F_{n,D}$ or $F_{D_n}$. 

One great advantage of quasisymmetric functions is that they
facilitate the use of standard in place of semi-standard objects,
allowing us to use a finite number of terms in the expression even
when there is an infinite number of variables. For example, we have
the following expansion for Schur functions due to
Gessel\cite{Gessel1984}.

\begin{theorem}[\cite{Gessel1984}]
  For $\lambda$ a partition of $n$, we have
  \begin{equation}
    s_{\lambda}(X) \ = \ \sum_{T \in \SYT(\lambda)} F_{\Des(T)}(X),
    \label{e:quasi_schur}
  \end{equation}
  where $\SYT(\lambda)$ denotes the set of all standard Young tableaux
  of shape $\lambda$.
\label{thm:quasi_schur}
\end{theorem}

For example, for $n = 4$, we have
\begin{eqnarray*}
s_{(3,1)} & = & F_{\{1\}} + F_{\{2\}} + F_{\{3\}} \\
s_{(2,2)} & = & F_{\{1,3\}} + F_{\{2\}} \\
s_{(2,1,1)} & = & F_{\{1,2\}} + F_{\{1,3\}} + F_{\{2,3\}} 
\end{eqnarray*}

Analogously, we may express the Schur $P$-functions in terms of the
fundamental basis.

\begin{proposition}
  For $\gamma$ a strict partition of $n$, we have
  \begin{eqnarray}
    Q_{\gamma} (X) & = & \sum_{S \in \siSYT(\gamma)} F_{\Des(S)}(X) = 2^{\ell(\gamma)} \sum_{S \in \siSYTp(\gamma)} F_{\Des(S)}(X)
    \label{e:quasi_Q}\\
    P_{\gamma} (X) & = & \sum_{S \in \siSYTp(\gamma)} F_{\Des(S)}(X),
    \label{e:quasi_P}
  \end{eqnarray}
  where $\siSYT(\gamma)$ denotes the set of all marked standard
  tableaux of shape $\gamma$, and $\siSYTp(\gamma)$ denotes the subset
  where no entry on the main diagonal is marked.
\label{prop:quasi_P}
\end{proposition}

\begin{proof}
  The formula follows from \eqref{e:schur_P} by standardizing the
  reading word while maintaining the positions of the marked
  letters. This is well-defined given the chosen total order.
\end{proof}

For example, signing the entries in Figure~\ref{fig:shSYT}, in all
possible ways gives 
$$P_{(3,1)} = F_{\{1\}} + 2F_{\{2\}} + F_{\{3\}} + F_{\{1,2\}} +
2F_{\{1,3\}} + F_{\{2,3\}}.$$ 

The similarity between \eqref{e:quasi_schur} and \eqref{e:quasi_P} is
the key to our proof of Theorem~\ref{thm:P_pos}. In particular, it is
easy to compute from the above expansion that
$$P_{(3,1)} = s_{(3,1)} + s_{(2,2)} + s_{(2,1,1)}.$$

However, notice that the summation in \eqref{e:quasi_P} is not over
\emph{standard} objects for type B. For this, we need a new family of
quasisymmetric functions.

For $P \subseteq \{2,3,\ldots,n-1\}$ with no consecutive entries,
define the shifted fundamental quasisymmetric function $G_P(X)$ by
\begin{equation}
  G_{P}(X) \ = \ \sum_{P \subseteq \Spike(D)} F_{D}(X),
  \label{e:quasisym_P}
\end{equation}
where the sum is over all subsets $D \subseteq \{1,2,\ldots,n-1\}$ for
which $\Spike(D)$ contains $P$. Again, when we wish to emphasize $n$,
we may write $G_{n,P}$ or $G_{P_n}$. For example, for $n=4$, we have
\begin{eqnarray*}
G_{\{2\}}(X) & = & F_{\{1\}}(X) + F_{\{2\}}(X) + F_{\{1,3\}}(X) + F_{\{2,3\}}(X), \\
G_{\{3\}}(X) & = & F_{\{2\}}(X) + F_{\{3\}}(X) + F_{\{1,2\}}(X) + F_{\{1,3\}}(X).
\end{eqnarray*}

The shifted fundamental quasisymmetric functions of degree $n$ form a
basis for a subspace of quasisymmetric functions of degree $n$ of
dimension the $n$th Fibonacci number. The shifted fundamental
quasisymmetric functions allow us to rewrite the Schur $P$-functions
as follows.

\begin{theorem}
  For $\gamma$ a strict partition of $n$, we have
  \begin{eqnarray}
    Q_{\gamma} (X) & = & \sum_{S \in \shSYT(\gamma)} 2^{|\Peak(S)|+1} G_{\Peak(S)}(X)
    \label{e:quasi_QH} \\
    P_{\gamma} (X) & = & 2^{-\ell(\gamma)} \sum_{S \in \shSYT(\gamma)} 2^{|\Peak(S)|+1} G_{\Peak(S)}(X),
    \label{e:quasi_PH}
  \end{eqnarray}
  where $\shSYT(\gamma)$ denotes the set of all standard shifted
  tableaux of shape $\gamma$, and $\Peak(S) = \Peak(\Des(S))$.
\label{thm:quasi_PH}
\end{theorem}

Note that for any $S \in \shSYT(\gamma)$, $|\Peak(S)| \geq
\ell(\gamma)-1$, so the expansion of $P_{\gamma}$ in terms of the
shifted fundamental quasisymmetric functions is integral. 

\begin{proof}
  Fix $S \in \shSYT(\gamma)$, and suppose $T \in \siSYT(\gamma)$ is
  such that removing the markings on $T$ gives $S$. If $i-1 \not\in
  \Des(S)$, then $i-1 \in \Des(T)$ if and only if $i$ is marked, and
  if $i-1 \in \Des(S)$, then $i-1 \in \Des(T)$ if and only if $i-1$ is
  unmarked. 

  First we claim that any $T \in \siSYT(\gamma)$ that gives $S$ when
  the markings are removed satisfies $\Peak(S) \subseteq
  \Spike(\Des(T))$. To see this, note that $i \in \Peak(S)$ if and
  only if both $i-1 \not\in \Des(S)$ and $i \in \Des(S)$. By the
  previous analysis, if $i$ is marked in $T$, then $i-1 \in \Des(T)$
  and $i \not\in \Des(T)$, and if $i$ is unmarked in $T$, then $i-1
  \not\in \Des(T)$ and $i \in \Des(T)$. Therefore $i \in
  \Spike(\Des(T))$.

  Next we claim that for any $D \subset [n-1]$ for which $\Peak(S)
  \subseteq \Spike(D)$, there are exactly $2^{|\Peak(S)|+1}$ standard marked
  tableaux $T$ that give $S$ when the markings are removed for which
  $\Des(T) = D$. Indeed, for $i \in \Peak(S)$, set $h = \min \{k | k <
  i \mbox{ and } k \not\in \Des(S)\}$ and set $j= \max \{k | k > i
  \mbox{ and } k \in \Des(S)\}$. Then, by the analysis above, $D$
  determines the markings for all $h < k \leq i$ and all $i \leq k <
  j$, but toggling the marking for $h$ or $j$ does not change
  $D$. If $i < i^{\prime}$ are consecutive entries of $\Peak(S)$, then
  the $j$ for $i$ and the $h$ for $i^{\prime}$ coincide. Thus there
  are exactly $|\Peak(S)|+1$ letters that can be marked or unmarked without
  affecting $D$.

  These two claims prove the expansion for $Q_{\gamma}$, and the
  result for $P_{\gamma}$ follows by \eqref{e:schur_P}.
\end{proof}

For example, using Figure~\ref{fig:shSYT}, we compute 
$$P_{(3,1)} = G_{\{2\}} + G_{\{3\}}.$$

%%%%%%%%%%%%%%%%%%%%%%%%%%%%%%%%%%%%%%%%%%%%%%%%%%%%%%%%%%%%%%%% 
\section{Dual equivalence and Schur positivity}
%%%%%%%%%%%%%%%%%%%%%%%%%%%%%%%%%%%%%%%%%%%%%%%%%%%%%%%%%%%%%%%%
\label{s:dual_equivalence}

Haiman \cite{Haiman1992} defined \emph{elementary dual equivalence
  involutions} on permutations as follows. If $a,b$ are two
consecutive letters of the word $w$, and $c$ is also consecutive with
$a,b$ and appears between $a$ and $b$ in $w$, then interchanging $a$
and $b$ is an elementary dual equivalence move. In this case, we refer
to $c$ as the \emph{witness} for the dual equivalence interchanging
$a$ and $b$. When $\{a,b,c\} = \{i-1,i,i+1\}$, we denote this
involution by $d_i$, and we regard words with $c$ not between $a$ and
$b$ as fixed points for $d_i$. For examples, see
Figure~\ref{fig:elementary}.

\begin{figure}[ht]
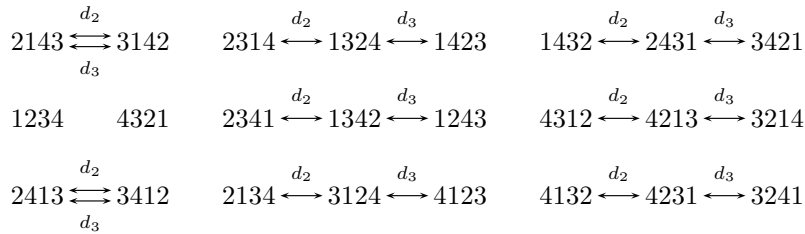

    \begin{displaymath}
      \begin{array}{c@{\hskip 2em}c@{\hskip 2em}c@{\hskip 2em}c@{\hskip 2em}c@{\hskip 2em}c@{\hskip 2em}c@{\hskip 2em}c}
        \rnode{a2}{2143} & \rnode{a3}{3142} & \rnode{b1}{2314} & \rnode{b2}{1324} & \rnode{b3}{1423} 
        & \rnode{b4}{1432} & \rnode{b5}{2431} & \rnode{b6}{3421} \\[4ex]
        \rnode{a1}{1234} & \rnode{a6}{4321} & \rnode{c1}{2341} & \rnode{c2}{1342} & \rnode{c3}{1243} 
        & \rnode{c4}{4312} & \rnode{c5}{4213} & \rnode{c6}{3214} \\[4ex]
        \rnode{a4}{2413} & \rnode{a5}{3412} & \rnode{d1}{2134} & \rnode{d2}{3124} & \rnode{d3}{4123} 
        & \rnode{d4}{4132} & \rnode{d5}{4231} & \rnode{d6}{3241}
      \end{array} 
      \everypsbox{\scriptstyle}
      \psset{nodesep=2pt,linewidth=.1ex}
      \ncline[offset=2pt]{<->} {a2}{a3} \naput{d_2}
      \ncline[offset=2pt]{<->} {a3}{a2} \naput{d_3}
      \ncline[offset=2pt]{<->} {a4}{a5} \naput{d_2}
      \ncline[offset=2pt]{<->} {a5}{a4} \naput{d_3}
      \ncline{<->} {b1}{b2} \naput{d_2}
      \ncline{<->} {b2}{b3} \naput{d_3}
      \ncline{<->} {b4}{b5} \naput{d_2}
      \ncline{<->} {b5}{b6} \naput{d_3}
      \ncline{<->} {c1}{c2} \naput{d_2}
      \ncline{<->} {c2}{c3} \naput{d_3}
      \ncline{<->} {c4}{c5} \naput{d_2}
      \ncline{<->} {c5}{c6} \naput{d_3}
      \ncline{<->} {d1}{d2} \naput{d_2}
      \ncline{<->} {d2}{d3} \naput{d_3}
      \ncline{<->} {d4}{d5} \naput{d_2}
      \ncline{<->} {d5}{d6} \naput{d_3}
    \end{displaymath}
    \caption{\label{fig:elementary}The dual equivalence
      classes of permutations of length $4$.}
\end{figure}

Two permutations $w$ and $u$ are \emph{dual equivalent} if there
exists a sequence $i_1,\ldots,i_k$ such that $u = d_{i_k} \cdots
d_{i_1}(w)$. Haiman \cite{Haiman1992} showed that the dual equivalence
involutions extend to standard Young tableaux via their reading words
and that dual equivalence classes correspond precisely to all standard
Young tableaux of a given shape, e.g. see Figure~\ref{fig:classes}.

\begin{figure}[ht]
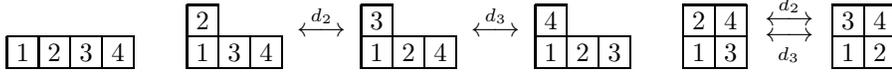

    \begin{displaymath}
      \tableau{ \\ 1 & 2 & 3 & 4}
      \hspace{2em}
      \tableau{2 \\ 1 & 3 & 4} \hspace{.3em} 
      \stackrel{d_2}{\longleftrightarrow} 
      \hspace{.3em} \tableau{3 \\ 1 & 2 & 4} \hspace{.3em}
      \stackrel{d_3}{\longleftrightarrow} 
      \hspace{.3em} \tableau{4 \\ 1 & 2 & 3} 
      \hspace{2em}
      \tableau{2 & 4 \\ 1 & 3} \hspace{.2em} 
      \begin{array}{c}
        \stackrel{d_2}{\longleftrightarrow} \\[-.5ex]
        \stackrel{\displaystyle\longleftrightarrow}{_{d_3}}
      \end{array}
      \hspace{.2em} \tableau{3 & 4 \\ 1 & 2}
    \end{displaymath}
    \caption{\label{fig:classes}Three dual equivalence classes of $\SYT$
      of size $4$.}
\end{figure}

Given this, we may rewrite \eqref{e:quasi_schur} in terms of dual
equivalence classes as
\begin{equation}
  s_{\lambda}(X) \ = \ \sum_{T \in [T_{\lambda}]} F_{\Des(T)}(X),
\label{e:classes}
\end{equation}
where $[T_{\lambda}]$ denotes the dual equivalence class of some fixed
$T_{\lambda} \in \SYT(\lambda)$. 

This paradigm shift to summing over objects in a dual equivalence
class is the basis for the universal method for proving that a
quasisymmetric generating function is symmetric and Schur positive
\cite{Assaf-DEGs-I}.  Motivated by \eqref{e:classes}, we have the
abstract notion of dual equivalence for any set of objects endowed
with a descent set.

Given involutions $\varphi_2,
\ldots, \varphi_{n-1}$ on a set $\mathcal{A}$ of combinatorial objects, for $1 < h < i < n$ we
consider the \emph{restricted dual equivalence class}
$[T]_{(h-1,i+1)}$ generated by $\varphi_h,\ldots,\varphi_i$. In
addition, we consider the \emph{restricted and shifted descent set}
$\mathrm{Des}_{(h,i)}(T)$ obtained by intersecting $\mathrm{Des}(T)$
with $\{h-1,\ldots,i\}$ and subtracting $h-2$ from each element so
that $\mathrm{Des}_{(h,i)}(T) \subseteq [i-h+2]$.

\begin{definition}[\cite{Assaf-DEGs-I}]
  Let $\mathcal{A}$ be a finite set, and let $\mathrm{Des}$ be a map
  on $\mathcal{A}$ such that $\mathrm{Des}(T) \subseteq [n-1]$ for all
  $T \in \mathcal{A}$. A \emph{dual equivalence for
    $(\mathcal{A},\mathrm{Des})$} is a family of involutions
  $\{\varphi_i\}_{1<i<n}$ on $\mathcal{A}$ such that

  \renewcommand{\theenumi}{\roman{enumi}}
  \begin{enumerate}
  \item For all $i-h \leq 3$ and all $T \in \mathcal{A}$, there
    exists a partition $\lambda$ of $i-h+3$ such that
    \[ \sum_{U \in [T]_{(h,i)}} F_{\mathrm{Des}_{(h,i)}(U)}(X) = s_{\lambda}(X). \] 
    
  \item For all $|i-j| \geq 3$ and all $T \in\mathcal{A}$, we have
    \begin{displaymath}
      \varphi_{j} \varphi_{i}(T) = \varphi_{i} \varphi_{j}(T).
    \end{displaymath}

  \end{enumerate}

  \label{def:strong}
\end{definition}

By \eqref{e:classes}, dual equivalence classes of tableaux precisely
correspond to Schur functions. Definition~\ref{def:strong} was
formulated so that the same property holds true for dual equivalence
classes for any pair $(\A,\Des)$. 

\begin{theorem}[\cite{Assaf-DEGs-I}]
  If $\{\varphi_i\}$ is a dual equivalence for $(\A,\Des)$, and $U \in
  \A$, then 
  \begin{equation}
    \sum_{T \in [U]} F_{\Des(T)}(X) \ = \ s_{\lambda}(X)
  \end{equation}
  for some partition $\lambda$. In particular, the quasisymmetric
  generating function for $\A$ is symmetric and Schur positive.
  \label{thm:positivity}
\end{theorem}

Our first goal, undertaken in Section~\ref{s:signed}, is to construct
dual equivalence involutions for the set $\siSYT(\gamma)$ of marked
standard shifted tableaux of shifted shape $\gamma$ with descent
function given by $\Des$. Our second goal, undertaken in
Section~\ref{s:shifted}, is to give a version of
Theorem~\ref{thm:positivity} that proves Schur $P$-positivity.

%%%%%%%%%%%%%%%%%%%%%%%%%%%%%%%%%%%%%%%%%%%%%%%%%%%%%%%%%%%%%%%%
\section{Dual equivalence for marked tableaux}
%%%%%%%%%%%%%%%%%%%%%%%%%%%%%%%%%%%%%%%%%%%%%%%%%%%%%%%%%%%%%%%%
\label{s:signed}

The objects for which we will construct a dual equivalence are the
marked standard tableaux, identified as shifted symmetric tableaux, with the associated descent function given by applying \eqref{e:des} to the reading word. Define the \emph{diagonal} of a cell to be its row index minus its column index. In particular, the main diagonal is zero, with the cells below positive and the reflected cells above negative.

\begin{definition}
  Let $S$ be a marked standard tableau. For $1<i<n$, let $a \leq b
  \leq c$ be the diagonals on which $i-1,i,i+1$
  reside. Then $\psi_i(w)$ is given by the following rule:
  \begin{itemize}
  \item if $i$ lies on diagonal $b$, then $\psi_i(S) = S$;
  \item else if $a=b$ (respectively $b=c$), then reflect the occupant
    on diagonal $c$ (respectively $a$) to lie on diagonal $-c$
    (respectively $-a$);
  \item else if $\left| |a|-|c| \right| = 1$, then reflect the
    occupant on diagonals $a$ and $c$ to lie on diagonal $-a$ and
    $-c$, respectively;
  \item else swap the occupants of diagonals $a$ and $c$.
  \end{itemize}
  \label{def:dual_shifted}
\end{definition}

For examples of $\psi$ on marked standard tableaux, see
Figure~\ref{fig:dual_shifted}. 

\begin{figure}[ht]
    \begin{displaymath}
      \begin{array}{c@{\hskip 3em}c@{\hskip 3em}c@{\hskip 3em}c}
      \rnode{t1}{\tableau{\cb \\ 2 & \cb & 3 \\ \cb & 1 & \cb & 4}} &
      \rnode{t2}{\tableau{\cb \\ \cb & \cb & 3 \\ \cb & 1 & 2 & 4}}         &
      \rnode{t3}{\tableau{\cb \\ \cb & \cb & 4 \\ \cb & 1 & 2 & 3}}         &
      \rnode{t4}{\tableau{\cb \\ 2 & \cb & 4 \\ \cb & 1 & \cb & 3}} \\[3\cellsize]
      \rnode{b1}{\tableau{4 \\ \cb & \cb & 3 \\ \cb & 1 & 2 & \cb}} &
      \rnode{b2}{\tableau{4 \\ 2 & \cb & 3 \\ \cb & 1 & \cb & \cb}} &
      \rnode{b3}{\tableau{3 \\ 2 & \cb & 4 \\ \cb & 1 & \cb & \cb}} &
      \rnode{b4}{\tableau{3 \\ \cb & \cb & 4 \\ \cb & 1 & 2 & \cb}} 
      \end{array}
      \psset{nodesep=2pt,linewidth=.1ex}
      \ncline{<->} {t1}{t2} \naput{\psi_2}
      \ncline{<->} {t2}{t3} \naput{\psi_3}
      \ncline[offset=2pt]{<->} {t4}{b4} \naput{\psi_2}
      \ncline[offset=2pt]{<->} {b4}{t4} \naput{\psi_3}
      \ncline{<->} {b1}{b2} \naput{\psi_2}
      \ncline{<->} {b2}{b3} \naput{\psi_3}
    \end{displaymath}
    \caption{\label{fig:dual_shifted}Dual equivalence for 
      $\siSYT(3,1)$.}
\end{figure}

\begin{proposition}
  The maps $\{\psi_i\}$ give well-defined involutions on the set of
  marked standard tableaux. Furthermore, if $S$ has no marked entries
  along the main diagonal, then neither does $\psi_i(S)$ for any $i$.
  \label{prop:well}
\end{proposition}

\begin{proof}
  Clearly when $S$ is a fixed point for $\psi_i$, it remains so after
  doing nothing. Suppose, then, that this is not the case. There are
  three cases to consider based on how $\psi_i$ acts. A key
  observation in what follows is that two of $a,b,c$ are equal in
  absolute value if and only if both $i-1$ and $i+1$ have diagonal
  $0$. Indeed, if this is not the case then there is a cell southeast
  of the diagonal in question containing an entry $x$ that lies
  strictly between two of $i-1,i,i+1$. This forces $x=i$, so the
  entries of the two cells in question are $i-1$ and $i+1$. Moreover,
  if this is not the main diagonal, then there exists a cell northwest
  of the diagonal in question containing an entry $y \neq i$ with $i-1
  < y < i+1$, which is not possible.

  First, suppose that $a=b<c$. From the observation above, we must
  have $0 = a = b < c$, in which case $-c < a = b < c$. Therefore
  $\psi_i$ toggles the marking on the entry in $c$, and so is an
  involution for this case and the related case $a < b=c$.

  Second, suppose $a<b<c$ and $\left| |a| - |c| \right| = 1$. The
  latter constraint means that in the standard shifted tableau
  obtained by removing markings, the entries corresponding to $a$ and
  $c$ form a vertical or horizontal domino. In particular, the former
  assumption ensures that neither $a$ nor $c$ can be $0$. Therefore we
  have $a < 0 < c$, and so $-c < b < -a$. Again, $\psi_i$ preserves
  the case and acts as an involution.

  Finally, suppose $a<b<c$ and $\left| |a| - |c| \right| > 1$. In this
  case, $\psi_i$ clearly preserves the case, so we need only show that
  the result is a marked standard tableau. Since $i$ does not have
  diagonal $b$, the entries being swapped are consecutive, implying that
  the only potential row or column violation is with the two swapped
  entries. The latter condition above ensures that, in the standard
  shifted tableau obtained by removing the markings, the swapped
  entries do not lie in the same row or column, and so no violations
  can result.
\end{proof}

\begin{theorem}
  The maps $\{\psi_i\}_{1<i<n}$ give a dual equivalence for marked
  standard tableaux. In particular, the Schur $Q$- and $P$-functions
  are Schur positive.
  \label{thm:strong_pos}
\end{theorem}

\begin{proof}
  The definition of $\psi_i$ depends only on the positions of the
  entries $i-1,i,i+1$. If $|i-j| \geq 3$, then the relevant entries
  are disjoint, and so $\psi_i$ and $\psi_j$ must commute. Therefore
  by Definition~\ref{def:strong}, it is enough to show that the
  equivalence classes are single Schur functions for all marked
  standard tableaux of skew shape $\gamma \setminus \delta$, where
  $\delta \subset \gamma$ and $|\gamma|-|\delta| \leq 6$. From the
  definition of $\psi$, if two consecutive diagonals have no cells,
  then we may collapse them to one empty diagonal without changing the
  structure of the equivalence class. In particular, we may assume
  $\gamma \subset (11,9,7,5,3,1)$. More precisely, the number of skew
  shapes to check for $n=3,4,5,6$ is $10,31,98,307$, respectively. These cases can be (and have been) checked by a patient hand or by a computer.
\end{proof}

The first half of Theorem~\ref{thm:P_pos} now follows from
Theorems~\ref{thm:positivity} and \ref{thm:strong_pos}.

\begin{corollary}
  For $\gamma$ a strict partition, we have
  \begin{equation}
    P_{\gamma}(X) = \sum_{\lambda} g_{\gamma,\lambda} s_{\lambda}(X),
  \end{equation}
  where $g_{\gamma,\lambda}$ is the number of dual equivalence classes
  of marked standard tableaux of shape $\gamma$ under the action of
  $\{\psi_i\}$ that are isomorphic to $\SYT(\lambda)$. In particular,
  $g_{\gamma,\lambda} \in \mathbb{N}$.
  \label{cor:P_pos}
\end{corollary}

%%%%%%%%%%%%%%%%%%%%%%%%%%%%%%%%%%%%%%%%%%%%%%%%%%%%%%%%%%%%%%%%
\section{Shifted dual equivalence graphs}
%%%%%%%%%%%%%%%%%%%%%%%%%%%%%%%%%%%%%%%%%%%%%%%%%%%%%%%%%%%%%%%%
\label{s:shifted}

Haiman \cite{Haiman1992} also defined \emph{elementary shifted dual
  equivalence involutions} on permutations as follows. If $a,b$ are
two consecutive letters of the word $w$, $c$ is also consecutive with
$a,b$ and appears between $a$ and $b$ in $w$, and $d$ is also
consecutive with $a,b,c$ and appears left of $c$ in $w$, then
interchanging $a$ and $b$ is an elementary shifted dual equivalence
move. In this case, we again refer to $c$ as the \emph{witness}, and
we refer to $d$ as the \emph{bystander} for the shifted dual
equivalence interchanging $a$ and $b$. When $\{a,b,c,d\} =
\{i-1,i,i+1,i+2\}$, we denote this involution by $\sd_i$, and we
regard words with $c$ not between $a$ and $b$ or with $d$ right of $c$
as fixed points for $\sd_i$. This rule is illustrated in
Figure~\ref{fig:elementary_sh}.

\begin{figure}[ht]
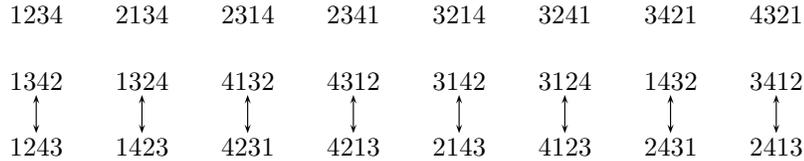

    \begin{displaymath}
      \begin{array}{c@{\hskip 2em}c@{\hskip 2em}c@{\hskip 2em}c@{\hskip 2em}c@{\hskip 2em}c@{\hskip 2em}c@{\hskip 2em}c}
        \rnode{f1}{1234} & \rnode{f2}{2134} & \rnode{f3}{2314} & \rnode{f4}{2341} & 
        \rnode{f5}{3214} & \rnode{f6}{3241} & \rnode{f7}{3421} & \rnode{f8}{4321} \\[3ex] 
        \rnode{t1}{1342} & \rnode{t2}{1324} & \rnode{t3}{4132} & \rnode{t4}{4312} & 
        \rnode{t5}{3142} & \rnode{t6}{3124} & \rnode{t7}{1432} & \rnode{t8}{3412} \\[3ex]        
        \rnode{b1}{1243} & \rnode{b2}{1423} & \rnode{b3}{4231} & \rnode{b4}{4213} & 
        \rnode{b5}{2143} & \rnode{b6}{4123} & \rnode{b7}{2431} & \rnode{b8}{2413}
      \end{array} 
      \everypsbox{\scriptstyle}
      \psset{nodesep=2pt,linewidth=.1ex}
      \ncline{<->} {t1}{b1} 
      \ncline{<->} {t2}{b2} 
      \ncline{<->} {t3}{b3} 
      \ncline{<->} {t4}{b4} 
      \ncline{<->} {t5}{b5} 
      \ncline{<->} {t6}{b6} 
      \ncline{<->} {t7}{b7} 
      \ncline{<->} {t8}{b8} 
    \end{displaymath}
    \caption{\label{fig:elementary_sh}The shifted dual equivalence
      classes of permutations of length $4$.}
\end{figure}

Haiman \cite{Haiman1992} showed that the shifted dual equivalence
involutions extend to standard shifted tableaux via their reading
words and that dual equivalence classes correspond precisely to all
standard shifted tableaux of a given shape. For examples, see
Figures~\ref{fig:classes_sh} and \ref{fig:SG7}.

\begin{figure}[ht]
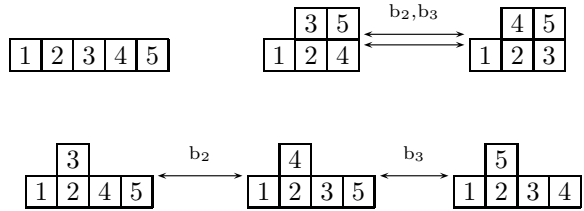

    \begin{displaymath}
      \begin{array}{c@{\hskip 2.5\cellsize}c@{\hskip 2.5\cellsize}c}
        \rnode{s1}{\tableau{\\ 1 & 2 & 3 & 4 & 5}} &
        \rnode{d1}{\tableau{& 3 & 5 \\ 1 & 2 & 4}} &
        \rnode{d2}{\tableau{& 4 & 5 \\ 1 & 2 & 3}} \\[3\cellsize]
        \rnode{t1}{\tableau{& 3 \\ 1 & 2 & 4 & 5}} &
        \rnode{t2}{\tableau{& 4 \\ 1 & 2 & 3 & 5}} &
        \rnode{t3}{\tableau{& 5 \\ 1 & 2 & 3 & 4}} 
      \end{array}
      \everypsbox{\scriptstyle}
      \psset{nodesep=2pt,linewidth=.1ex}
      \ncline[offset=2pt]{<->} {d1}{d2} \naput{\sd_2,\sd_3}
      \ncline[offset=2pt]{<->} {d2}{d1} 
      \ncline{<->} {t1}{t2} \naput{\sd_2}
      \ncline{<->} {t2}{t3} \naput{\sd_3}
    \end{displaymath}
    \caption{\label{fig:classes_sh}The shifted dual equivalence
      classes of $\shSYT$ of size $5$.}
\end{figure}

Comparing Figure~\ref{fig:classes_sh} with Figure~\ref{fig:classes},
it might seem that shifted dual equivalence classes are the same as
dual equivalence classes. However, shifted classes can have triple
edges, whereas dual equivalence classes can have at most double edges,
so the equality is an artifact of small numbers. To make this
statement precise, we introduce the notion of a morphism between
dual equivalences.

\begin{definition}
  Let $\A,\mathcal{B}$ be two sets of combinatorial objects, and let
  $\Des_A$ and $\Des_B$ be descent maps on each. Given involutions
  $\alpha_i$ on $\A$ and $\beta_i$ on $\mathcal{B}$, a \emph{morphism}
  from $(\A,\Des_A,\alpha)$ to $(\mathcal{B},\Des_B,\beta)$ is a map
  $\phi:\A \rightarrow \mathcal{B}$ such that for every $a \in \A$, we
  have $\Des_{A}(a) = \Des_{B}(\phi(a))$ and $\phi(\alpha_i(a)) =
  \beta_i(\phi(a))$. A morphism is an \emph{isomorphism} if it is a
  bijection from $\A$ to $\mathcal{B}$.
  \label{def:morphism}
\end{definition}

To avoid cumbersome notation, we omit the subscript for $\Des$ when it
is clear from context.

Since dual equivalence pertains to descent sets and shifted dual
equivalence pertains to peak sets, we also need to shift a peak set as
follows. Given a subset $D$ of positive integers greater than $1$, let
$D-1$ be the subset obtained by subtracting $1$ from each element of
$D$.

\begin{proposition}
  For nonnegative integers $r>s$, the shifted dual equivalence for
  $(\shSYT((r,s)), \Peak-1)$ given by $\{\sd_i\}$ is isomorphic to the
  dual equivalence on $(\SYT((r-1,s)), \Des)$ given by $\{d_i\}$. For
  $\gamma$ a strict partition with more than $2$ parts, the shifted
  dual equivalence on $(\shSYT(\gamma), \Peak-1)$ given by
  $\{\sd_i\}$ is not isomorphic to $(\SYT(\lambda),\Des)$ given by
  $\{d_i\}$ for any partition $\lambda$.
\end{proposition}

\begin{proof}
  Consider the map $\phi$ from $\shSYT((r,s))$ to $\SYT((r-1,s))$
  given by removing the cell containing $1$, subtracting $1$ from each
  entry. On the level of sets, $\phi$ is clearly a bijection. One
  easily checks that, in addition, $\Peak(T)-1 = \Des(\phi(T))$, and
  $\phi(\sd_{i+1}(T)) = d_i(\phi(T))$. Therefore $\phi$ is an
  isomorphism of dual equivalences.

  The shifted tableau $T$ of shape $\gamma = (3,2,1)$ with reading
  word $645123$ has $\sd_2 = \sd_3 = \sd_4$. Any strict partition with
  at least $3$ parts must contain $\gamma$, and so it contains an
  element that restricts to $T$. In particular, such an element has
  $\sd_2 = \sd_3 = \sd_4$, and so the equivalence cannot be isomorphic
  to any dual equivalence on standard Young tableaux.
\end{proof}

Completely analogous to the unshifted case, for
$\gamma$ a strict partition, we may rewrite \eqref{e:quasi_PH} in
terms of dual equivalence classes as
\begin{eqnarray}
  Q_{\gamma}(X) & = & \sum_{T \in [T_{\gamma}]} 2^{|\Peak(T)|+1} G_{\Peak(T)}(X)
  \label{e:classes_shQ} \\
  P_{\gamma}(X) & = & 2^{-\ell(\gamma)} \sum_{T \in [T_{\gamma}]} 2^{|\Peak(T)|+1} G_{\Peak(T)}(X),
  \label{e:classes_sh}
\end{eqnarray}
where $[T_{\gamma}]$ denotes the shifted dual equivalence class of
some fixed $T_{\gamma} \in \shSYT(\gamma)$. 

By \eqref{e:classes_sh}, shifted dual equivalence classes of standard
shifted tableaux precisely correspond to Schur $Q$-functions or Schur
$P$-functions, depending on the chosen scaling.  Following the
analogy, our goal is to use this paradigm shift to summing over
objects in a shifted dual equivalence class to give a universal method
for proving that a quasisymmetric generating function is symmetric and
Schur $Q$-positive or Schur $P$-positive.

Since the subset statistic for this case is the peak set instead of
the descent set, we make the following notation. Given involutions $\varphi_2, \ldots, \varphi_{n-2}$, for $1 < j < i <
n-1$ on a set $\A$ of combinatorial objects, we consider the restricted shifted dual equivalence classes
$[T]_{(j,i)}$ generated by $\varphi_j,\ldots,\varphi_i$. In addition,
we consider the restricted and shifted peak sets $\Peak_{(j,i)}(T)$
obtained by intersecting $\Peak(T)$ with $\{j-1,\ldots,i+1\}$ and
subtracting $j-2$ from each element so that $\Peak_{(j,i)}(T)
\subseteq [i-j+1]$.

\begin{definition}
  Let $\A$ be a finite set, and let $\Peak$ be a peak map on $\A$ such
  that $\Peak(T) \subseteq \{2,\ldots,n-1\}$ with no consecutive
  entries for all $T \in \A$. A \emph{shifted dual equivalence for
    $(\A,\Peak)$} is a family of involutions $\{\varphi_i\}_{1<i<n-1}$
  on $\A$ such that

  \renewcommand{\theenumi}{\roman{enumi}}
  \begin{enumerate}
  \item For all $|i-j| \leq 5$ and all $T \in \mathcal{A}$, there
    exists a strict partition $\gamma$ of $|i-j|+4$ such that
    \[ \sum_{U \in [T]_{(j,i)}} 2^{|\Peak_{(j,i)}(U)|+1} G_{\mathrm{Peak}_{(j,i)}(U)}(X) = Q_{\gamma}(X). \] 
    
  \item For all $|i-j| \geq 4$ and all $T \in\mathcal{A}$, we have
    \begin{displaymath}
      \varphi_{j} \varphi_{i}(T) = \varphi_{i} \varphi_{j}(T).
    \end{displaymath}

  \end{enumerate}

  \label{def:strong_sh}
\end{definition}

Definition~\ref{def:strong_sh} completely characterizes shifted dual
equivalence classes in the same sense that Definition~\ref{def:strong}
characterizes dual equivalence classes. 

\begin{proposition}
  For $\gamma$ a strict partition of $n$, the involutions
  $\{\sd_i\}_{1<i<n-1}$ give a shifted dual equivalence for
  $\shSYT(\gamma)$ with peak function given by \eqref{e:peaks}. 
  \label{prop:shifted_DEG}
\end{proposition}

\begin{proof}
  Condition (i) holds by \eqref{e:classes_sh}, and condition (ii)
  follows from the fact that $\sd_i$ depends only on the relative
  positions of $i-1,i,i+1,i+2$ and for $|i-j| \geq 4$ these sets are
  disjoint.
\end{proof}

The real purpose of Definition~\ref{def:strong_sh} is to establish the
following analog of Theorem~\ref{thm:positivity}.

\begin{theorem}
  If $\{\varphi_i\}$ is a shifted dual equivalence for $(\A,\Des)$ and
  $U \in \A$, then
  \begin{equation}
    \sum_{T \in [U]} 2^{|\Peak(T)|+1} G_{\Peak(T)}(X) \ = \ Q_{\gamma}(X)
  \end{equation}
  for some strict partition $\gamma$. In particular, the
  quasisymmetric generating function for $\A$ is symmetric and Schur
  $Q$-positive.
  \label{thm:positivity_sh}
\end{theorem}

We prove Theorem~\ref{thm:positivity_sh} along the same lines as the
structure theorem for dual equivalence given in
\cite{Assaf-DEGs-I}. To begin, we show that the shifted dual
equivalences for standard shifted tableaux are pairwise nonisomorphic
with no nontrivial automorphisms. This is the shifted analog of
\cite{Assaf-DEGs-I}[Proposition~3.9].

\begin{figure}[ht]
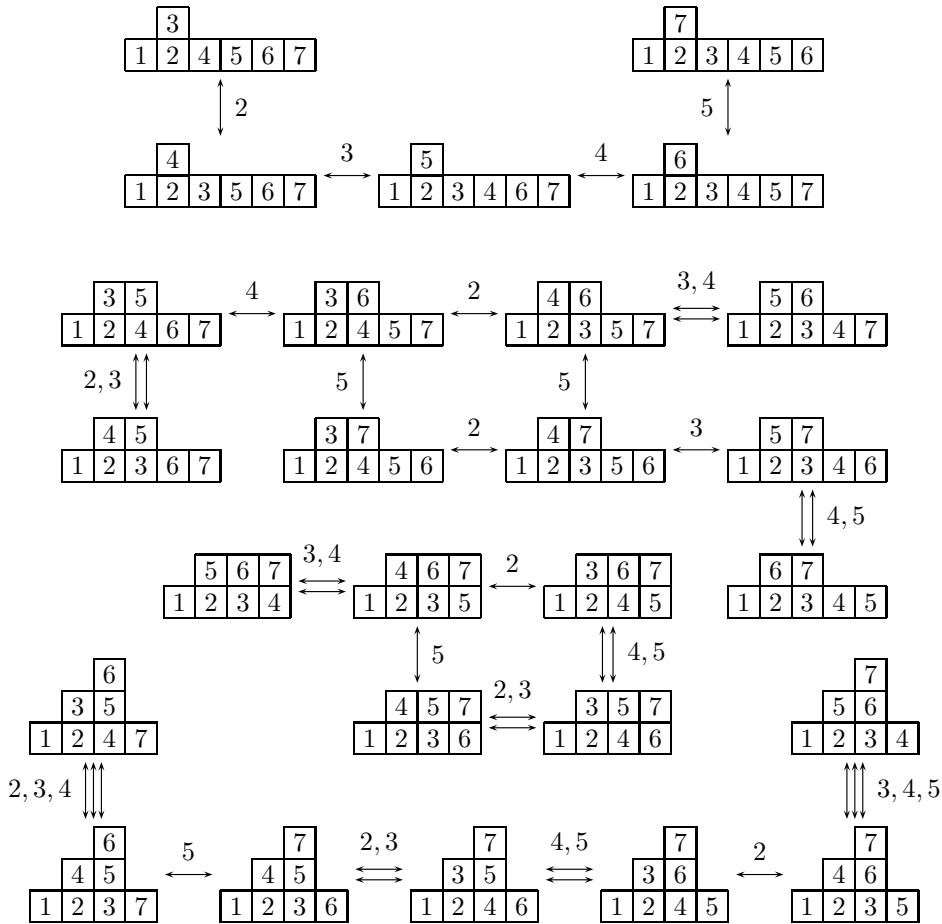

  \begin{displaymath}
    \begin{array}{c}
      \begin{array}{c@{\hskip 2\cellsize}c@{\hskip 2\cellsize}c}
        \rnode{a1}{\tableau{& 3 \\ 1 & 2 & 4 & 5 & 6 & 7}} & & 
        \rnode{a5}{\tableau{& 7 \\ 1 & 2 & 3 & 4 & 5 & 6}} \\[3\cellsize]
        \rnode{a2}{\tableau{& 4 \\ 1 & 2 & 3 & 5 & 6 & 7}} & 
        \rnode{a3}{\tableau{& 5 \\ 1 & 2 & 3 & 4 & 6 & 7}} &
        \rnode{a4}{\tableau{& 6 \\ 1 & 2 & 3 & 4 & 5 & 7}} 
      \end{array} \\[5\cellsize]
      \begin{array}{c@{\hskip 2\cellsize}c@{\hskip 2\cellsize}c@{\hskip 2\cellsize}c}
        \rnode{b8}{\tableau{& 3 & 5\\ 1 & 2 & 4 & 6 & 7}} &
        \rnode{b6}{\tableau{& 3 & 6\\ 1 & 2 & 4 & 5 & 7}} &
        \rnode{b5}{\tableau{& 4 & 6\\ 1 & 2 & 3 & 5 & 7}} &
        \rnode{b7}{\tableau{& 5 & 6\\ 1 & 2 & 3 & 4 & 7}} \\[3\cellsize]
        \rnode{b9}{\tableau{& 4 & 5\\ 1 & 2 & 3 & 6 & 7}} &
        \rnode{b4}{\tableau{& 3 & 7\\ 1 & 2 & 4 & 5 & 6}} &
        \rnode{b3}{\tableau{& 4 & 7\\ 1 & 2 & 3 & 5 & 6}} &
        \rnode{b2}{\tableau{& 5 & 7\\ 1 & 2 & 3 & 4 & 6}} \\[3\cellsize]
        & & & \rnode{b1}{\tableau{& 6 & 7\\ 1 & 2 & 3 & 4 & 5}}
      \end{array} \\[-2\cellsize]
      \begin{array}{c@{\hskip 2\cellsize}c@{\hskip 2\cellsize}c@{\hskip 4\cellsize}}
        \rnode{c1}{\tableau{& 5 & 6 & 7 \\ 1 & 2 & 3 & 4}} &
        \rnode{c2}{\tableau{& 4 & 6 & 7 \\ 1 & 2 & 3 & 5}} &
        \rnode{c3}{\tableau{& 3 & 6 & 7 \\ 1 & 2 & 4 & 5}} \\[3\cellsize] &
        \rnode{c4}{\tableau{& 4 & 5 & 7 \\ 1 & 2 & 3 & 6}} &
        \rnode{c5}{\tableau{& 3 & 5 & 7 \\ 1 & 2 & 4 & 6}}
      \end{array} \\[-3\cellsize]
      \begin{array}{c@{\hskip 2\cellsize}c@{\hskip 2\cellsize}c@{\hskip 2\cellsize}c@{\hskip 2\cellsize}c}
        \rnode{d7}{\tableau{& & 6 \\ & 3 & 5 \\ 1 & 2 & 4 & 7}} & & & &
        \rnode{d1}{\tableau{& & 7 \\ & 5 & 6 \\ 1 & 2 & 3 & 4}} \\[4\cellsize]
        \rnode{d6}{\tableau{& & 6 \\ & 4 & 5 \\ 1 & 2 & 3 & 7}} &
        \rnode{d5}{\tableau{& & 7 \\ & 4 & 5 \\ 1 & 2 & 3 & 6}} &
        \rnode{d4}{\tableau{& & 7 \\ & 3 & 5 \\ 1 & 2 & 4 & 6}} &
        \rnode{d3}{\tableau{& & 7 \\ & 3 & 6 \\ 1 & 2 & 4 & 5}} &
        \rnode{d2}{\tableau{& & 7 \\ & 4 & 6 \\ 1 & 2 & 3 & 5}} 
      \end{array} 
    \end{array}
    \psset{nodesep=3pt,linewidth=.1ex}
    \ncline{<->} {a1}{a2} \naput{2}
    \ncline{<->} {a2}{a3} \naput{3}
    \ncline{<->} {a3}{a4} \naput{4}
    \ncline{<->} {a4}{a5} \naput{5}
    \ncline[offset=2pt]{<->} {b9}{b8} \naput{2,3}
    \ncline[offset=2pt]{<->} {b8}{b9}
    \ncline{<->} {b8}{b6} \naput{4}
    \ncline[offset=2pt]{<->} {b5}{b7} \naput{3,4}
    \ncline[offset=2pt]{<->} {b7}{b5} 
    \ncline{<->} {b6}{b5} \naput{2}
    \ncline{<->} {b4}{b6} \naput{5}
    \ncline{<->} {b3}{b5} \naput{5}
    \ncline{<->} {b4}{b3} \naput{2}
    \ncline{<->} {b3}{b2} \naput{3}
    \ncline[offset=2pt]{<->} {b1}{b2} 
    \ncline[offset=2pt]{<->} {b2}{b1} \naput{4,5}
    \ncline[offset=2pt]{<->} {c1}{c2} \naput{3,4}
    \ncline[offset=2pt]{<->} {c2}{c1} 
    \ncline{<->} {c2}{c3} \naput{2}
    \ncline{<->} {c2}{c4} \naput{5}
    \ncline[offset=2pt]{<->} {c3}{c5} \naput{4,5}
    \ncline[offset=2pt]{<->} {c5}{c3} 
    \ncline[offset=2pt]{<->} {c4}{c5} \naput{2,3}
    \ncline[offset=2pt]{<->} {c5}{c4} 
    \ncline[offset=3pt]{<->} {d7}{d6} 
    \ncline{<->} {d7}{d6} 
    \ncline[offset=3pt]{<->} {d6}{d7} \naput{2,3,4}
    \ncline{<->} {d6}{d5} \naput{5}
    \ncline[offset=2pt]{<->} {d5}{d4} \naput{2,3}
    \ncline[offset=2pt]{<->} {d4}{d5} 
    \ncline[offset=2pt]{<->} {d4}{d3} \naput{4,5}
    \ncline[offset=2pt]{<->} {d3}{d4} 
    \ncline{<->} {d3}{d2} \naput{2}
    \ncline[offset=3pt]{<->} {d2}{d1} 
    \ncline{<->} {d2}{d1} 
    \ncline[offset=3pt]{<->} {d1}{d2} \naput{3,4,5}
  \end{displaymath}
\caption{\label{fig:SG7}The standard shifted dual equivalences
  of size $7$.}
\end{figure}

\begin{proposition} 
  For $\gamma,\delta$ strict partitions, if $\phi: \shSYT(\gamma)
  \rightarrow \shSYT(\delta)$ is an isomorphism of shifted dual
  equivalences, then $\gamma = \delta$ and $\phi=\mathrm{id}$.
\label{prop:noauto-noniso}
\end{proposition}

\begin{proof}
  Let $T_{\gamma}$ be the standard shifted tableau obtained by filling
  the numbers 1 through $n$ into the rows of $\gamma$ from left to
  right, bottom to top. For any standard shifted tableau $T$ such that
  $\Peak(T) = \Peak(T_{\gamma})$, the numbers $1$ through $\lambda_1$
  must form a horizontal strip since $1\not\in\Des(T_{\gamma})$ and
  $\lambda_1$ is the smallest element of
  $\Peak(T_{\gamma})$. Similarly, $\lambda_1 + 1$ through $\lambda_1 +
  \lambda_2$ must form a horizontal strip, and so on.  In particular,
  if $\Peak(T) = \Peak(T_{\gamma})$ for some $T$ of shape $\delta$, then
  $\gamma \leq \delta$ in dominance order with equality if and only if $T = T_{\gamma}$.

  Suppose $\phi : \shSYT(\gamma) \rightarrow \shSYT(\delta)$ is an
  isomorphism. Let $T_{\gamma}$ be as above for $\gamma$, and let
  $T_{\delta}$ be the corresponding tableau for $\delta$. Then since
  $\Peak(\phi(T_{\gamma})) = \Peak(T_{\gamma})$, $\gamma \leq
  \delta$. Conversely, since $\Peak(\phi^{-1}(T_{\delta})) =
  \Peak(T_{\delta})$, $\delta \leq \gamma$.  Therefore $\gamma=\delta$.
  Furthermore, $\phi(T_{\gamma}) = T_{\gamma}$. For $T \in
  \shSYT(\gamma)$ such that $\sd_i(T_{\gamma}) = T$, we have
  $\phi(T) = \sd_i(T_{\gamma}) = T$.  Extending this, every shifted
  tableau connected to a fixed point by some sequence of elementary
  shifted dual equivalences is also a fixed point for $\phi$, hence
  $\phi = \mathrm{id}$ on each $\shSYT(\gamma)$.
\end{proof}

Thus far we have avoided using the language of signed, colored graphs
to describe shifted dual equivalence. The following result is
characterizing the local structure of shifted dual equivalence classes
analogous to the original definition of dual equivalence graphs in
\cite{Assaf-DEGs-I}. 

\begin{lemma}
  Let $\{\varphi_i\}_{1<i<n-1}$ be a shifted dual equivalence for
  $(\shSYT(\gamma),\Peak)$ for some strict partition $\gamma$ of $n$
  with $n \leq 7$. Then $\varphi_i = \sd_i$. In particular, there is a
  unique shifted dual equivalence for standard shifted tableaux of
  size at most $7$.
  \label{lem:small}
\end{lemma}

\begin{proof}
  Given $\gamma$ a strict partition of size $n \leq 7$, no two
  standard shifted tableaux of shape $\gamma$ have the same peak
  set. This is easy to observe from Figure~\ref{fig:SG7}, for example.
  The result is trivial for $\gamma = (n)$ or for $n \leq 3$ since
  there are no nontrivial involutions, so we have four cases to
  consider, and we may assume $\gamma \neq (n)$.

  For $n=4$ (see Figure~\ref{fig:shSYT}), the only case left to
  consider is $\gamma = (3,1)$ which has two standard shifted tableaux
  which must necessarily be paired by $\varphi_2$ to ensure they lie
  in a single equivalence class. By condition (i) for degree $4$, this
  is enough to characterize fixed points for $\varphi_i$ as those
  elements $T$ with $i,i+1\not\in\Peak(T)$
  (cf. \cite{Assaf-DEGs-I}[Definition 3.2, axiom 1]). Moreover, $i \in
  \Peak(T)$ if and only if $i+1 \in \Peak(\varphi_i(T))$
  (cf. \cite{Assaf-DEGs-I}[Definition 3.2, axiom 2]). Armed with this,
  the result now follows for $\gamma = (n-1,1)$ since each standard
  shifted tableaux has a unique peak.

  For $n=5$ (see Figure~\ref{fig:classes_sh}), the only case left to
  consider is $\gamma = (3,2)$ which has two standard shifted tableaux
  which must necessarily be paired by $\varphi_2$ and by $\varphi_3$
  to ensure they lie in a single equivalence class. In particular,
  condition (i) for degree $5$ now implies that when
  $\varphi_i(T) = \varphi_{i+1}(T)$ the cardinality of $\Peak(T) \cap
  \{i,i+1,i+2\}$ changes (cf. \cite{Assaf-DEGs-I}[Definition 3.2,
    axiom 3]).

  For $n=6$, there are two cases to consider. For $\gamma = (3,2,1)$,
  there are two standard shifted tableaux, neither of which is a fixed
  point for $\varphi_i$ for $i=2,3,4$. Therefore they must be
  connected by a triple edge. For $\gamma = (4,2)$, the five standard
  shifted tableaux have peak sets $\{3\}, \{2,4\}, \{2,5\},
  \{3,5\}, \{4\}$. Thus for each $i=2,3,4$, exactly one of the
  five standard shifted tableaux is a fixed point for $\varphi_i$, so
  given the toggling condition on peak sets, there are two possible
  pairing for each $\varphi_i$, $i=2,3,4$. If $\varphi_3$ pairs the
  tableaux with peak sets $\{3\}$ and $\{4\}$, then it must also pair
  the tableaux with peak sets $\{3,5\}$ and $\{2,4\}$. By the analysis
  for $n=5$ and condition (i) for degree $5$, this means that both
  $\varphi_2$ and $\varphi_4$ pair the tableaux with peak sets
  $\{3,5\}$ and $\{2,4\}$ as well. But then these two tableaux are in
  an equivalence class of their own, contradicting that classes are
  (locally) Schur $Q$-positive. Thus $\varphi_3$ pairs tableaux with
  peak sets $\{3\}$ and $\{2,4\}$ and also tableaux with
  peak sets $\{3,5\}$ and $\{4\}$. By the same logic, the former case
  is also paired by $\varphi_2$, and the latter by $\varphi_4$. This
  correctly constrains the structure.

  Finally, for $n=7$ (see Figure~\ref{fig:SG7}), there are three cases
  to consider, and the analysis is analogous to the case for $n=6$,
  where now we may appeal to the case $n=6$ as well. 
\end{proof}

Lemma~\ref{lem:small} states that the shifted dual equivalence on
standard shifted tableaux of size up to $7$ is unique. This result is
tight since there exist standard shifted tableaux of size $8$ with the
same shape and the same peak set. 

Given strict partitions $\gamma,\delta$ with $\gamma \subset \delta$, fix
a filling of the cells of $\delta\setminus\gamma$ with $|\gamma|+1,\ldots,|\delta|$, say $B$. Let
$\shSYT(\gamma,B) \subset \shSYT(\delta)$ be subset of shifted standard
tableaux that restrict to $B$ when skewed by $\gamma$. The resulting
shifted dual equivalence on $\shSYT(\gamma,B)$ has the same
involutions as $\shSYT(\gamma)$, but the $\Peak$ function has now
been extended. With this in mind, we show that for $\gamma$ a
partition of $n$, any extension of the peak function for
$\shSYT(\gamma)$ can be modeled by $\shSYT(\gamma,B)$ for some
augmenting tableau $B$. The following result is the shifted analog of
\cite{Assaf-DEGs-I}[Lemma 3.11].

\begin{lemma}
  Let $\{\varphi_i\}_{1<i<n}$ be a shifted dual equivalence for
  $(\A,\Peak)$. Let $T \in \A$ such that there exists an isomorphism
  $\phi$ from $([T]_{(2,n-2)}, \Peak\cap\{2,3,\ldots,n-1\})$ to
  $(\shSYT(\gamma),\Peak)$ for some strict partition $\gamma$ of
  $n$. Then there exists a unique standard shifted tableau $B$ of
  shape $\delta/\gamma$, $\delta$ a strict partition of size $n+1$, such
  that $\phi$ gives an isomorphism from $([T]_{(2,n-2)} ,\Peak)$ to
  $(\shSYT(\gamma,B),\Peak)$.
\label{lem:extend-signs}
\end{lemma}

\begin{proof}
  The result follows for $n \leq 6$ by Lemma~\ref{lem:small}, so
  assume $n\geq 6$. The restricted equivalence classes of
  $\shSYT(\gamma)$ under $\sd_2, \ldots, \sd_{n-3}$ may each be
  identified with a strict partition of $n-1$ contained in $\gamma$,
  or, equivalently, with the unique outer corner of $\gamma$
  containing the entry $n$ for each tableau of the restricted
  equivalence class. Therefore the isomorphism from $([T]_{(2,n-2)},
  \Peak\cap\{2,3,\ldots,n-1\})$ to $(\shSYT(\gamma),\Peak)$ allows us
  to identify each restricted equivalence classes of $[T]_{(2,n-2)}$
  under the maps $\varphi_2,\ldots,\varphi_{n-3}$ with an outer corner
  of $\gamma$. Given a restricted class $\C \subseteq [T]_{(2,n-2)}$,
  mark $\C$ with a $-$ if some element of $\C$ has a peak at $n$, and
  mark $\C$ with a $+$ if no element of $\C$ has a peak at $n$. We
  claim that if $\C$ and $\D$ are two restricted equivalence classes
  with the corner of $\C$ above the corner of $\D$, then if $\C$ is
  marked $-$, so is $\D$, and if $\D$ is marked $+$ then so is
  $\C$. That is to say, we have the situation depicted in
  Figure~\ref{fig:monotonic}. This being the case, letting $B$ be the
  augmenting tableau of shape $\gamma$ with $n+1$ added to the inner
  corner above which $n$ cannot be a peak and below which $n$ can be a
  peak, we have that $([T]_{(2,n-2)} ,\Peak)$ is isomorphic to
  $(\shSYT(\gamma,B),\Peak)$.

  \begin{figure}[ht]
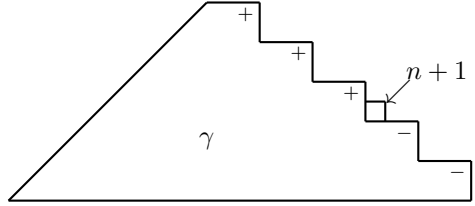

    \begin{center}
      \psset{xunit=2em}
      \psset{yunit=1.5em}
      \pspicture(-4,0)(5,5)
      \psline(1,4)(1,5)
      \psline(2,3)(2,4)
      \psline(3,2)(3,3)
      \psline(4,1)(4,2)
      \psline(5,0)(5,1)
      \psline(-3.75,0)(5,0)
      \psline(-3.75,0)(0,5)
      \psline(4,1)(5,1)
      \psline(3,2)(4,2)
      \psline(2,3)(3,3)
      \psline(1,4)(2,4)
      \psline(0,5)(1,5)
      \rput(0.75,4.7){$_+$}
      \rput(1.75,3.7){$_+$}
      \rput(2.75,2.7){$_+$}
      \rput(3.75,1.7){$_-$}
      \rput(4.75,0.7){$_-$}
      \psline(3,2.5)(3.375,2.5)
      \psline(3.375,2)(3.375,2.5)
      \rput(3.625,2.75){$\swarrow$}
      \rput(4.35,3.25){$n+1$}
      \rput(0,1.5){$\gamma$}
      \endpspicture
      \caption{\label{fig:monotonic} Identifying the unique position
        for $n+1$ based on the peak set.}
    \end{center}
  \end{figure}

  To prove the claim, assume for contradiction that the corner of $\C$
  is above the corner of $\D$ and that $\C$ is marked $-$ and $\D$ is
  marked $+$. We subclaim that the corner of $\C$ cannot be on the
  main diagonal. For $U \in \C$ with a peak at $n$,
  Lemma~\ref{lem:small} ensures that the restricted class
  $[U]_{(n-4,n-1)}$ is isomorphic to $\shSYT(\varepsilon)$ for a unique
  strict partition $\varepsilon$ of size $7$. If $U$ maps to $A\in
  \shSYT(\varepsilon)$ under this isomorphism, then $[U]_{(n-4,n-2)}$ is
  isomorphic to $[A]_{(2,4)}$. In particular, $A$ has a peak at $6$,
  ensuring that $5$ is weakly above $6$ in $A$. By the uniqueness in
  Lemma~\ref{lem:small}, the image of $U$ in $\shSYT(\gamma)$ must
  also map isomorphically to $A$ when restricted to entries
  $n-5,\ldots,n$ and maps $\sd_{n-4},\sd_{n-3},\sd_{n-2}$. In
  particular, the two must have the same descent set, forcing $n-1$ to
  be weakly above $n$. Thus $n$ cannot be in a top row of length $1$. 

  Choose a shifted tableau $D$ of shape $\gamma$ with specified
  positions for $n,n-1,\ldots,n-5$ as follows. Put $n$ in the outer
  corner corresponding to $\D$, $n-1$ in the outer corner
  corresponding to $\C$, and $n-2$ between these two corners. This
  placement for $n-2$ is always possible since there must be at least
  one cell that is both left of $n-1$ and below $n$ in order for both
  to be outer corners. Furthermore, by the prior assertion that the
  corner for $\C$ is not on the main diagonal, we may place $n-3$ left
  of $n-1$. Therefore if we removed $n-1$ and $n$, then $n-3$ and
  $n-2$ would both occupy outer corners. Given this, we may place
  $n-4$ below $n-3$ and left of $n-2$. Finally, since $\gamma$ is
  still strict, there is at least one cell left of $n-4$, so we may
  place $n-5$ there. See Figure~\ref{fig:six}.

  \begin{figure}[ht]
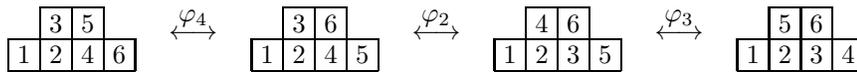

    \begin{displaymath}
      \tableau{ & 3 & 5 \\ 1 & 2 & 4 & 6}
      \hspace{1em} \stackrel{\displaystyle\varphi_{4}}{\longleftrightarrow} \hspace{1em}
      \tableau{ & 3 & 6 \\ 1 & 2 & 4 & 5}
      \hspace{1em} \stackrel{\displaystyle\varphi_{2}}{\longleftrightarrow} \hspace{1em}
      \tableau{ & 4 & 6 \\ 1 & 2 & 3 & 5}
      \hspace{1em} \stackrel{\displaystyle\varphi_{3}}{\longleftrightarrow} \hspace{1em}
      \tableau{ & 5 & 6 \\ 1 & 2 & 3 & 4}
    \end{displaymath}
    \caption{\label{fig:six}Relative positions for $n-5,\ldots,n$ in
      $D$ (far left) and $C$ (far right).}
  \end{figure}

  With $D$ as described, $\sd_{n-2}$ acts on $D$ by interchanging
  $n-1$ and $n$ with witness $n-2$ and bystander $n-3$. Then
  $\sd_{n-4}$ acts on $\sd_{n-2}(D)$ by interchanging $n-3$ and $n-2$
  with witness $n-4$ and bystander $n-5$. Set $C =
  \sd_{n-3}\sd_{n-4}\sd_{n-2}(D)$, where $\sd_{n-3}$ acts by
  interchanging $n-2$ and $n-1$ with witness $n-3$ and bystander
  $n-4$. Thus the element of $\D$ that maps to $D$, say $d$, and
  the element of $\C$ that maps to $C$, say $c$, lie in the same class
  under $\varphi_{n-4},\varphi_{n-3},\varphi_{n-2}$. Finally, we use
  the assumption on peak sets and Lemma~\ref{lem:small} for
  $[c]_{(n-4,n-1)} = [d]_{(n-4,n-1)}$ to see that $d$ must correspond to
  degree $7$ shape $(5,2)$ whereas $c$ must correspond to degree $7$
  shape $(4,2,1)$, which are not the same.
\end{proof}

The main result that will establish Theorem~\ref{thm:positivity_sh} is
the following shifted analog of \cite{Assaf-DEGs-I}[Theorem 3.12].

\begin{theorem}
  Let $\{\varphi_i\}_{1<i<n}$ be a shifted dual equivalence for
  $(\A,\Peak)$ with a single equivalence class, and suppose that for
  each $T \in \A$ there exists an isomorphism from
  $([T]_{(2,n-2)},\Peak\cap\{2,3,\ldots,n-1\})$ to
  $(\shSYT(\gamma),\Peak)$ for some strict partition $\gamma$ of
  $n$. Then there exists a morphism from $(\A,\Peak)$ to
  $(\shSYT(\varepsilon), \Peak)$ for a unique strict partition $\varepsilon$
  of $n+1$ containing $\gamma$.
\label{thm:cover}
\end{theorem}

\begin{proof}
  For $n < 7$, the result follows from Lemmas~\ref{lem:extend-signs}
  and \ref{lem:small}.  If $\varphi_{n-1}$ acts trivially on the
  entire equivalence class, then $\Peak \equiv \varnothing$ and
  $\varepsilon = (n+1)$.  Therefore we proceed by induction on $n$ and
  assuming that $\A$ has at least one element (thus two elements) not
  fixed by $\varphi_{n-1}$.

  By Lemma~\ref{lem:extend-signs}, the isomorphism from
  $([T]_{(2,n-2)},\Peak\cap\{2,3,\ldots,n-1\})$ to
  $(\shSYT(\gamma),\Peak)$ extends to an isomorphism from
  $([T]_{(2,n-2)},\Peak)$ to $(\shSYT(\gamma,B),\Peak)$ for a unique
  augmenting shifted tableau $B$, say with shape $\varepsilon/\gamma$. We
  will show that for any $[T]_{(2,n-2)}$, the shape $\varepsilon$ is the
  same and that we may glue these isomorphisms together to obtain a
  morphism from $(\A,\Peak)$ to $(\shSYT(\varepsilon), \Peak)$.

  \begin{figure}[ht]
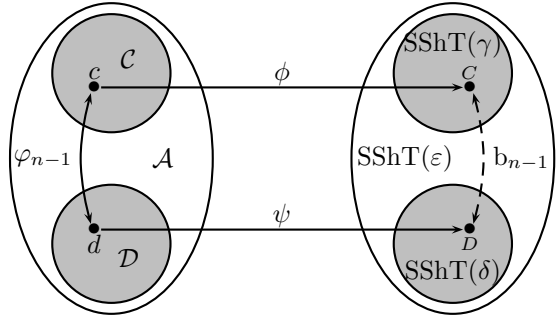

    \begin{center}
      \psset{xunit=3ex}
      \psset{yunit=2.5ex}
      \pspicture(0,0)(14,10)
      \pscircle[fillstyle=solid,fillcolor=lightgray](12,2){.8}
      \rput(12.5,2.5){$\bullet$}
      \rput(12.5,2){$_D$}
      \rput(12,1){$\shSYT(\delta)$}
      \pscircle[fillstyle=solid,fillcolor=lightgray](12,8){.8}
      \rput(12.5,7.5){$\bullet$}
      \rput(12.5,8){$_C$}
      \rput(12,9){$\shSYT(\gamma)$}
      \psellipse(12,5)(3,5.5)
      \rput(10.6,5){$\shSYT(\varepsilon)$}
      \pscurve[linestyle=dashed]{<->}(12.6,2.7)(12.8,3.75)(12.9,5)(12.8,6.25)(12.6,7.3)
      \rput(14,5){$\sd_{n-1}$}
      \pscircle[fillstyle=solid,fillcolor=lightgray](2,2){.8}
      \rput(1.5,2.5){$\bullet$}
      \rput(1.5,2){$d$}
      \rput(2.5,1.5){$\mathcal{D}$}
      \pscircle[fillstyle=solid,fillcolor=lightgray](2,8){.8}
      \rput(1.5,7.5){$\bullet$}
      \rput(1.5,8){$c$}
      \rput(2.5,8.5){$\C$}
      \psellipse(2,5)(3,5.5)
      \rput(3.5,5){$\A$}
      \pscurve{<->}(1.4,2.7)(1.2,3.75)(1.1,5)(1.2,6.25)(1.4,7.3)
      \rput(0,5){$\varphi_{n-1}$}
      \psline{->}(1.7,7.5)(12.3,7.5)
      \rput(7,8){$\phi$}
      \psline{->}(1.7,2.5)(12.3,2.5)
      \rput(7,3){$\psi$}
      \endpspicture
      \caption{\label{fig:extend} An illustration of the gluing
        process.}
    \end{center}
  \end{figure}

  Suppose $d = \varphi_{n-1}(c)$. Let $\C = [c]_{(2,n-2)}$ and $\D =
  [d]_{(2,n-2)}$. Let $\phi$ (resp. $\psi$) be the isomorphism from
  $\C$ (resp. $\D$) to $\shSYT(\gamma)$ (resp. $\shSYT(\delta)$), and
  set $C = \phi(c)$; see Figure~\ref{fig:extend}. We will show that
  $\psi(d) = \sd_{n-1}(C)$, and hence if $\gamma,B$ has shape
  $\varepsilon$, then so does $\delta,B$, and the maps $\phi$ and $\psi$
  glue together to give an morphism from $\C \cup \D$ to
  $\shSYT(\varepsilon)$ that commutes with $\varphi_{n-1}$ and
  $\sd_{n-1}$. There are two cases to consider, based on the relative
  positions of $n-1,n,n+1$ in $C$, regarded as a shifted tableau of
  shape $\varepsilon$.

  First suppose that $\sd_{n-1}$ fixes $n+1$. We will show that, in
  this case, $\C = \D$. There are three subcases to consider based on
  how $\sd_{n-1}$ acts on $C$ and $D$. For each, the idea is to use
  $\sd_2,\ldots,\sd_{n-5}$ to move to a tableau $C'$ for which
  $\sd_{n-1}(C') = \sd_{n-2}(C') = D'$, and then move from $D'$ back
  to $D$ by reversing the sequence of $\sd_2,\ldots,\sd_{n-5}$ using
  condition (ii) of Definition~\ref{def:strong_sh}; see
  Figure~\ref{fig:internal}.

  If $\sd_{n-1}$ interchanges $n-2$ and $n-1$, then $n$ must be the
  witness and $n+1$ the bystander, so there must exist a cell between
  $n-2$ and $n-1$ containing an entry at most $n-3$. Let $C'$ be the
  shifted tableau with $n-2,n-1,n,n+1$ in the same position as in $C$
  and with $n-3$ between $n-2$ and $n-1$. Then $\varphi_{n-2}$ also
  acts on $C'$ by exchanging $n-2$ and $n-1$, where $n-3$ and $n$ are
  the witness and bystander in some order.

  If $\sd_{n-1}$ interchanges $n-1$ and $n$ with $n-2$ as the witness
  and $n+1$ the bystander, then there must be a cell between $n+1$ and
  $n-2$ that contains an entry at most $n-3$. Let $C'$ be the shifted
  tableau with $n-2,n-1,n,n+1$ in the same position as in $C$ and with
  $n-3$ between $n-2$ and $n+1$. Then $\varphi_{n-2}$ also acts on
  $C'$ by exchanging $n-1$ and $n$, where $n-2$ is again the witness
  and $n-3$ is the bystander.

  If $\sd_{n-1}$ interchanges $n-1$ and $n$ with $n+1$ as the witness
  and $n-2$ the bystander, then again there must be a cell between
  $n+1$ and $n-2$ that contains an entry at most $n-3$. Let $C''$ be
  the shifted tableau with $n-2,n-1,n,n+1$ in the same position as in
  $C$ and with $n-3$ between $n-2$ and $n+1$. Now we go further,
  noting that there must be a cell between $n-2$ and $n-3$ containing
  an entry at most $n-4$, and we insist that $n-4$ lie in this cell in
  $C''$ as well. Finally, there must be a cell left of this that
  contains an entry at most $n-5$, and we insist that $n-5$ lie in
  this cell in $C''$ as well. Observe that the positions of
  $n-5,\ldots,n+1$ ensure that $[C'']_{(n-4,n-1)}$ is isomorphic to
  $\shSYT(4,3)$; see Figure~\ref{fig:SG7}. In particular, in this case
  $\sd_{n-4}$ commutes with $\sd_{n-1}$. Thus we may set $C' =
  \sd_{n-4}(C'')$, which interchanges $n-3$ and $n-2$. In this case,
  $\varphi_{n-2}$ acts on $C'$ by exchanging $n-1$ and $n$, where
  $n-2$ has been positioned as the witness and $n-3$ as the bystander.

  \begin{figure}[ht]
    \begin{center}
      \begin{displaymath}
        \begin{array}{\cs{5}\cs{6}\cs{5}\cs{6}\cs{3}\cs{3}\cs{3}\cs{6}\cs{5}\cs{6}c}
          \rnode{C}{\C} & \rnode{w}{c} & & \rnode{wl}{\B} & & \rnode{wm}{\cdots} 
          & & \rnode{wr}{\B} & & \rnode{wd}{c'} & \\[4ex]
          & & \Rnode{x}{d} & & \rnode{xl}{\B} & & \rnode{xm}{\cdots} 
          & & \rnode{xr}{\B} & & \rnode{xd}{d'} \\[7ex]
          & \Rnode[vref=0.5ex]{T}{C} & & \rnode{Tl}{\B} & & \rnode{Tm}{\cdots} 
          & & \rnode{Tr}{\B} & & \Rnode[vref=0.5ex]{Td}{C'} & \\[4ex]
          \rnode{G}{\shSYT(\gamma)} & & \Rnode[vref=0.5ex]{U}{D} & & \rnode{Ul}{\B} 
          & & \rnode{Um}{\cdots} & & \rnode{Ur}{\B} & & \Rnode[vref=0.5ex]{Ud}{D'} 
        \end{array}
        \psset{nodesep=3pt,linewidth=.1ex}
        \ncline{<->} {w}{wl}   \naput{\leq n-5}
        \ncline{<->} {wl}{wm}  %\naput{}
        \ncline{<->} {wm}{wr}  %\naput{}
        \ncline{<->} {wr}{wd}  \naput{\leq n-5}
        \ncline{<->} {x}{xl}   \nbput{\leq n-5}
        \ncline{<->} {xl}{xm}  %\naput{}
        \ncline{<->} {xm}{xr}  %\naput{}
        \ncline{<->} {xr}{xd}  \nbput{\leq n-5}
        \ncline[linestyle=dashed]{<->} {w}{x} \nbput{n-1}
        \ncline[linestyle=dashed]{<->} {wl}{xl} \nbput{n-1}
        \ncline[linestyle=dashed]{<->} {wr}{xr} \nbput{n-1}
        \ncline[offset=2pt]{<->} {wd}{xd} \naput {n-2}
        \ncline[linestyle=dashed,offset=2pt]{<->} {xd}{wd} \naput{n-1}
        \ncline[linewidth=.2ex,nodesep=6pt]{->} {C}{G}   \nbput{\phi}
        \ncline{<->} {T}{Tl}   \naput{\leq n-5}
        \ncline{<->} {Tl}{Tm}  %\naput{}
        \ncline{<->} {Tm}{Tr}  %\naput{}
        \ncline{<->} {Tr}{Td}  \naput{\leq n-5}
        \ncline{<->} {U}{Ul}   \nbput{\leq n-5}
        \ncline{<->} {Ul}{Um}  %\naput{}
        \ncline{<->} {Um}{Ur}  %\naput{}
        \ncline{<->} {Ur}{Ud}  \nbput{\leq n-5}
        \ncline[linestyle=dashed]{<->} {T}{U} \nbput{n-1}
        \ncline[linestyle=dashed]{<->} {Tl}{Ul} \nbput{n-1}
        \ncline[linestyle=dashed]{<->} {Tr}{Ur} \nbput{n-1}
        \ncline[offset=2pt]{<->} {Td}{Ud} \naput {n-2}
        \ncline[linestyle=dashed,offset=2pt]{<->} {Ud}{Td} \naput{n-1}
      \end{displaymath}
    \end{center}
    \caption{\label{fig:internal} The path from $C$ to $D$ in
      $\shSYT(\gamma)$ and its lift in $\C$.}
  \end{figure}

  In all cases, we have shown that $D \in [C]_{2,n-2}$. Therefore we
  can lift this to $\A$ to conclude that $d \in [c]_{2,n-2} = \C$. Set
  $c' = \phi^{-1}(C')$ and $d' = \psi^{-1}(D')$. By
  Lemma~\ref{lem:small}, there is a unique shifted dual equivalence
  structure on $[c']_{n-4,n-1}$ and on $[C']_{n-4,n-1}$, which forces
  $\varphi_{n-1}(c') = \phi^{-1}(\sd_{n-1}(C')) = \phi^{-1}(D') =
  d'$. Thus we may invoke condition (ii) of
  Definition~\ref{def:strong_sh} to conclude that $d =
  \varphi_{n-1}(c) = \phi^{-1} (\sd_{n-1}(C)) = \phi^{-1}(D)$. In
  this case $\C=\D$ and, by Proposition~\ref{prop:noauto-noniso},
  $\psi=\phi$.

  For the second case, we may assume $\sd_{n-1}$ acts on $C$ by
  interchanging $n$ and $n+1$ with $n-1$ as the witness and $n-2$ as
  the bystander. Consider the subset of $\shSYT(\gamma,B)$ with $n$
  and $n+1$ fixed in the same position as in $C$ and $n-1$ lying
  anywhere between them and $n-2$ anywhere left of $n-1$. In terms of
  the equivalences, for $n-1$ in a fixed position, these are all
  tableaux reachable using $\sd_h$ with $h \leq n-5$ and certain
  instances of $\sd_{n-4}$. We will return soon to the question of
  which instances these are. For now, let $\overline{[C]}$ denote the
  union of the equivalence classes $[C']_{(2,n-5)}$, where $C'$ has
  $n-1$ between $n$ and $n+1$ and $n-2$ somewhere left of $n-1$. The
  range of positions for $n-2$ determines the range of positions for
  $n-1$, which in turn uniquely determines the cells containing $n$
  and $n+1$, and so this set uniquely determines
  $\varepsilon$. Furthermore, which of $n,n+1$ occupies which cell is
  determined by whether or not $n$ is a peak. Therefore $\phi^{-1}$
  lifts $\overline{[C]}$ to a subset of to $\C$ that completely
  determines $\varepsilon$ as well as the positions of $n$ and $n+1$
  in the image of this set under $\phi$. We will show that the
  corresponding lifted subset for $\D$ is isomorphic but with $n$
  toggled into or out of the peak set.
  
  \begin{figure}[ht]
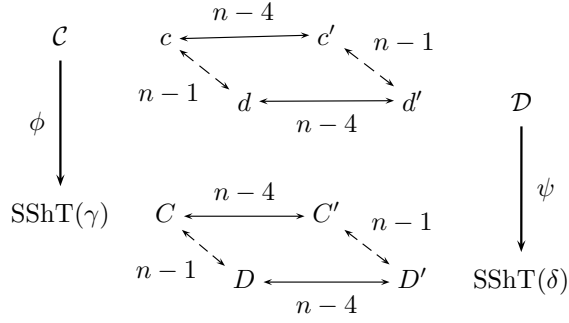

    \begin{center}
      \begin{displaymath}
        \begin{array}{\cs{4}\cs{5}\cs{5}\cs{5}\cs{4}c}
          \rnode{C}{\C} & \rnode{w}{c} & & \rnode{wd}{c'} & & \\[3ex]
          & & \rnode{x}{d} & & \rnode{xd}{d'} & \rnode{D}{\D} \\[7ex]
          \rnode{A}{\shSYT(\gamma)} & \Rnode[vref=0.5ex]{T}{C} & & \Rnode[vref=0.5ex]{Td}{C'} & & \\[3ex]
          & & \Rnode[vref=0.5ex]{U}{D} & & \Rnode[vref=0.5ex]{Ud}{D'} & \rnode{B}{\shSYT(\delta)} 
        \end{array}
        \psset{nodesep=3pt,linewidth=.1ex}
        \ncline{<->} {w}{wd}   \naput{n-4}
        \ncline{<->} {x}{xd}   \nbput{n-4}
        \ncline{<->} {T}{Td}   \naput{n-4}
        \ncline{<->} {U}{Ud}   \nbput{n-4}
        \ncline[linestyle=dashed]{<->} {w}{x} \nbput{n-1}
        \ncline[linestyle=dashed]{<->} {wd}{xd} \naput{n-1}
        \ncline[linestyle=dashed]{<->} {T}{U} \nbput{n-1}
        \ncline[linestyle=dashed]{<->} {Td}{Ud} \naput{n-1}
        \ncline[linewidth=.2ex,nodesep=6pt]{->} {C}{A}   \nbput{\phi}
        \ncline[linewidth=.2ex,nodesep=6pt]{->} {D}{B}   \naput{\psi}
      \end{displaymath}
    \end{center}
    \caption{\label{fig:typeC} Illustration of $\sd_{n-4}$ on
      $\overline{[C]} \cup \overline{[D]}$ and their lifts in $\C \cup
      \D$.}
  \end{figure}

  To prove the assertion, we return to the question of which instances
  of $\varphi_{n-4}$ are allowed in generating $\overline{[C]}$. For
  this, we may consider the structure of the degree $7$ equivalence
  class generated by $\varphi_{n-4},\ldots,\varphi_{n-1}$ which, by
  Lemma~\ref{lem:small}, have the structure of one of the components
  depicted in Figure~\ref{fig:SG7}. Examining the possibilities, it
  follows that whenever $\varphi_{n-4}$ keeps $n-2$ left of $n-1$, the
  map also commutes with $\varphi_{n-1}$, as depicted in
  Figure~\ref{fig:typeC}. By condition (ii) of shifted dual
  equivalences, $\varphi_h$ also commutes with $\varphi_{n-1}$ for $h
  \leq n-5$. Therefore all involutions used to generate
  $\overline{[C]}$ commute with $\varphi_{n-1}$. Thus $\varphi_{n-1}$
  gives an isomorphism from $\phi^{-1}(\overline{[C]})$ to a subset of
  $\D$. Let $D = \psi(d)$. Then $\overline{[D]} =
  \psi(\varphi_{n-1}(\phi^{-1}(\overline{[C]})))$. Since $\phi,\psi$
  and $\varphi_{n-1}$ are isomorphisms, $\overline{[D]}$ is isomorphic
  to $\overline{[C]}$, though exactly one has $n$ as a peak. By the
  earlier characterization of $\overline{[C]}$, this implies that the
  tableaux in $\overline{[D]}$ have shape $\varepsilon$, with the
  cells containing $n$ and $n+1$ reversed from that in
  $\overline{[C]}$. In particular, $\overline{[D]} =
  \sd_{n-1}(\overline{[C]})$, that is to say, $\phi$ and $\psi$
  glue to give a morphism from $\C \cup \D \subset \A$ to
  $\shSYT(\gamma,A) \cup \shSYT(\delta,B) \subset \shSYT(\varepsilon)$
  that respects $\varphi_{n-1}$. Therefore this map lifts to a
  morphism from $\A$ to $\shSYT(\varepsilon)$.
\end{proof}

Thus far we have used condition (i) of shifted dual equivalence only
for restricted equivalence classes of degree up to $7$. If this weaker
condition is all that is used, then Theorem~\ref{thm:cover} proves
that there is a morphism from $(\A,\Peak)$ to
$\shSYT(\varepsilon)$. While this morphism is always surjective, in
order to show that it is injective we must invoke condition (i) for
degree $9$.

\begin{theorem}
  Given any shifted dual equivalence $\{\varphi_i\}_{1<i<n}$ for
  $(\A,\Peak)$ with a single equivalence class, there exists a unique
  strict partition $\gamma$ of size $n+1$ such that there is an
  isomorphism of shifted dual equivalences between $(\A, \Peak)$ under
  $\{\varphi_i\}_{1<i<n}$ and $(\shSYT(\gamma), \Peak)$ under
  $\{\sd_i\}_{1<i<n}$.
  \label{thm:isomorphic_sh}
\end{theorem}

\begin{proof}
  We proceed by induction on $n+1$, noting that the result follows
  from Lemma~\ref{lem:small} for $n+1 \leq 7$. The involutions
  $\{\varphi_i\}_{1<i<n-1}$ give a shifted dual equivalence for any
  restricted equivalence class $[T]_{(2,n-2)}$, and so, by induction,
  each such class is isomorphic to $(\shSYT(\delta), \Peak)$ for a
  unique strict partition $\delta$ of $n$. By Theorem~\ref{thm:cover},
  there exists a morphism, say $\phi$, from $\A$ to $\shSYT(\gamma)$
  for a unique strict partition $\gamma$ of $n+1$. Surjectivity
  follows from the fact that $\A$ has a single equivalence class under
  $\{\varphi_i\}_{1<i<n}$ and that $\phi$ commutes with $\varphi_i$
  and $\sd_i$. To prove that $\phi$ is injective, we first claim that
  the fiber over each standard shifted tableau has the same
  cardinality.

  To prove the claim, we show that for any restricted equivalence
  class $\C$ under $\{\varphi_i\}_{1<i<n-1}$, say with $\phi(\C) =
  \shSYT(\delta)$, and any strict partition $\varepsilon \subset \gamma$
  of size $n$, there is a unique restricted equivalence class $\D$
  under $\{\varphi_i\}_{1<i<n-1}$ with $\phi(\D) = \shSYT(\varepsilon)$
  such that $d = \varphi_{n-1}(c)$ for some $c \in \C$ and some $d \in
  \D$. Once established, this gives a bijective correspondence between
  equivalence classes in $\phi^{-1}(\shSYT(\delta))$ and in
  $\phi^{-1}(\shSYT(\varepsilon))$, thus proving the result.

  To prove existence, if $\varepsilon \neq \delta$, let $C$ be a shifted
  standard tableau of shape $\gamma$ with $n+1$ in position
  $\gamma/\delta$, $n$ in position $\gamma/\varepsilon$, $n-1$ lying
  between, and $n-2$ lying left of $n-1$. Otherwise let $C$ be a
  standard shifted tableau with $n+1$ in position $\gamma/\delta$ and
  $n$ and $n-1$ lying on opposite sides, again with $n-2$ left of
  $n-1$. Let $c$ be the unique element in $\phi^{-1}(C) \cap \C$. Then
  $\phi(\varphi_{n-1}(c)) = \varphi_{n-1}(\phi(c)) \in \shSYT(\varepsilon)$.

  To prove uniqueness, let $d = \varphi_{n-1}(c)$ with $c \in \C$ and
  $d \in \D$. If $n+1$ lies between $n$ and $n-1$ in $\phi(C)$, then
  $\delta = \varepsilon$, and just as in the proof of
  Theorem~\ref{thm:cover}, we concluded that $\D = \C$ as
  desired. Alternately, assume $n-1$ lies between $n$ and $n+1$ in
  $\phi(c)$, and suppose $\varphi_{n-1}(c')=d'$ with $C' \in \C$ and
  $d' \in \D'$ where $\phi(D') = \shSYT(\varepsilon)$. Since $\phi(c)$
  and $\phi(c')$ have the same shape, and $\varphi_{n-1}(\phi(c)) =
  \phi(\varphi_{n-1}(c)) = \phi(d)$ and $\varphi_{n-1}(\phi(w')) =
  \phi(\varphi_{n-1}(w')) = \phi(d')$ have the same shape as
  well. Just as in the proof of Theorem~\ref{thm:cover}, $\phi(c)$ and
  $\phi(c')$ must lie in the same equivalence class under
  $\sd_2,\ldots,\sd_{n-5}$ and those instances of $\sd_{n-4}$ that
  commute with $\sd_{n-1}$. Lifting this class via $\phi$, $c$ and
  $c'$ must lie in the same equivalence class under
  $\varphi_2,\ldots,\varphi_{n-5}$ and those instances of
  $\varphi_{n-4}$ that commute with $\varphi_{n-1}$. Therefore
  applying $\varphi_{n-1}$, $d$ and $d'$ lie in the same equivalence
  class under $\varphi_2,\ldots,\varphi_{n-5}$ and those instances of
  $\varphi_{n-4}$ that commute with $\varphi_{n-1}$. In particular,
  $d' \in \D$, and so $\D = \D'$.

  Since each fiber has the same cardinality, say $k$, the generating
  function for $\A$ is $k Q_{\lambda}$. When $n \leq 9$, condition (i)
  ensures $k=1$ and the map to $\shSYT(\gamma)$ is an isomorphism. In
  particular, we may assume the restricted equivalence classes under
  $\varphi_{n-6},\ldots,\varphi_{n-1}$ are isomorphic to some
  $\shSYT(\delta)$. Given two standard shifted tableaux $C,D$ of the
  same shape, there exist tableaux $C',D'$ such that $C' \in
  [C]_{(2,n-2)}$, $D' \in [D]_{(2,n-2)}$, and $C' = \varphi_{n-1}(D')$
  (cf. axiom $6$ for dual equivalence graphs in
  \cite{Assaf-DEGs-I}[Definition 3.2]).

  Suppose $T,U,V,X \in \A$, with $U = \varphi_{n-1}(T)$, $V =
  \varphi_{n-1}(X)$, and $U$ and $V$ lie in the same restricted
  equivalence class under $\varphi_2,\ldots, \varphi_{n-2}$. We will
  show that there exist $T^{\prime},X^{\prime}$ in the same restricted
  equivalence classes as $T,X$, respectively, such that $T^{\prime}$
  and $X^{\prime}$ lie in the same degree $9$ equivalence class under
  $\varphi_{n-6},\ldots,\varphi_{n-1}$; see Figure~\ref{fig:ax6}. By
  the earlier remark and the result for degree up to $9$, this implies
  that there exist $T'',X''$ such that $T'' \in [T']_{(n-6,n-2)}
  \subseteq [T]_{(2,n-2)}$, $X'' \in [X']_{(n-6,n-2)} \subseteq
  [X]_{(2,n-2)}$, and $T'' = \varphi_{n-1}(X'')$. In particular, we
  must have $k=1$.
            
  \begin{figure}[ht]
    \begin{displaymath}
      \begin{array}{c@{\hskip 3em}c@{\hskip 3em}c}
        \ovalnode{L}{ \begin{array}{c}
            \rnode{T}{T} \\ [5ex]
            \rnode{TT}{T^{\prime}}
          \end{array} } &
        \ovalnode{M}{ \begin{array}{c@{\hskip 2em}c@{\hskip 2em}c@{\hskip 2em}c@{\hskip 2em}c@{\hskip 2em}c}
            \rnode{U}{U} & & & & & \rnode{V}{V} \\[5ex]
            \rnode{UU}{U^{\prime}} & \rnode{A}{\bullet} & \rnode{B}{\bullet} & \rnode{C}{\bullet} & \rnode{D}{\bullet} & \rnode{VV}{V^{\prime}} 
          \end{array} } &
        \ovalnode{R}{ \begin{array}{c}
            \rnode{X}{X} \\[5ex]
            \rnode{XX}{X^{\prime}} 
          \end{array} }
      \end{array}
      \everypsbox{\scriptstyle}
      \psset{nodesep=3pt,linewidth=.1ex}
      \ncline[nodesepB=8pt]{<->} {T}{U}  \naput{\varphi_{n-1}}
      \ncline[nodesepB=8pt]{<->} {TT}{UU}\naput{\varphi_{n-1}}
      \ncline[nodesepA=8pt]{<->} {VV}{XX}\naput{\varphi_{n-1}}
      \ncline[nodesepA=8pt]{<->} {V}{X}  \naput{\varphi_{n-1}}
      \ncline{<->} {UU}{A} \naput{\varphi_{n-2}}
      \ncline{<->} {A}{B}  \naput{\varphi_{n-4}}
      \ncline{<->} {B}{C} \naput{\varphi_{n-6}}
      \ncline{<->} {C}{D}  \naput{\varphi_{n-4}}
      \ncline{<->} {D}{VV} \naput{\varphi_{n-2}}
    \end{displaymath}
    \caption{\label{fig:ax6} Establishing injectivity from degree $9$.}
  \end{figure}

  By the inductive hypothesis and Theorem~\ref{thm:cover}, we may
  identify $T,U,V,X$ with shifted tableaux of shape $\gamma$,
  $|\gamma| = n+1$, and, when restricted to entries up to $n$,
  $T,U,V,X$ have shapes $\delta,\varepsilon,\varepsilon,\zeta$,
  respectively, with $\delta,\varepsilon,\zeta$ distinct strict
  partitions contained in $\gamma$. Then $\gamma/\rho$ must be
  a corner (end of row, top of column) for $\rho =
  \delta,\varepsilon,\zeta$. Assume these cells appear with
  $\gamma/\delta$ northeast of $\gamma/\zeta$
  northeast of $\gamma/\varepsilon$, noting that the other
  orders can be resolved in a similar way. Let $T^{\prime}$ be any
  shifted standard tableaux of shape $\gamma$ with $n+1$ in position
  $\gamma/\delta$, $n$ in position
  $\gamma/\varepsilon$, $n-1$ in position
  $\gamma/\zeta$, $n-2$ between $n+1$ and $n-1$, $n-3$ left of $n-2$
  (which is possible since $\gamma$ is a shifted shape), $n-4$ between
  $n-1$ and $n$, $n-5$ between $n-2$ and $n-4$, $n-6$ between $n-3$
  and $n-5$, and $n-7$ left of $n-6$ (which is possible since $\gamma$
  is a shifted shape). See Figure~\ref{fig:nine} for an illustration. 

  \begin{figure}
    \begin{displaymath}
      \begin{array}{c@{\hskip 3em}c@{\hskip 3em}c@{\hskip 3em}c}
        \rnode{a}{\tableau{ & & 9 \\ & 5 & 6 & 7 \\ 1 & 2 & 3 & 4 & 8}} &
        \rnode{b}{\tableau{ & & 8 \\ & 5 & 6 & 7 \\ 1 & 2 & 3 & 4 & 9}} &
        \rnode{c}{\tableau{ & & 7 \\ & 5 & 6 & 8 \\ 1 & 2 & 3 & 4 & 9}} &
        \rnode{d}{\tableau{ & & 7 \\ & 4 & 6 & 8 \\ 1 & 2 & 3 & 5 & 9}} \\[12ex]
        \rnode{h}{\tableau{ & & 6 \\ & 3 & 5 & 9 \\ 1 & 2 & 4 & 7 & 8}} &
        \rnode{g}{\tableau{ & & 6 \\ & 3 & 5 & 8 \\ 1 & 2 & 4 & 7 & 9}} &
        \rnode{f}{\tableau{ & & 7 \\ & 3 & 5 & 8 \\ 1 & 2 & 4 & 6 & 9}} &
        \rnode{e}{\tableau{ & & 7 \\ & 3 & 6 & 8 \\ 1 & 2 & 4 & 5 & 9}} 
      \end{array}
      \everypsbox{\scriptstyle}
      \psset{nodesep=3pt,linewidth=.1ex}
      \ncline{<->} {a}{b}  \naput{\sd_{n-1}}
      \ncline{<->} {b}{c}  \naput{\sd_{n-2}}
      \ncline[offset=4pt]{<->} {c}{d} \naput{\sd_{n-4}}
      \ncline[offset=0pt]{<->} {c}{d}  
      \ncline[offset=-4pt]{<->} {c}{d} \nbput{\overset{\scriptstyle\sd_{n-5}}{\sd_{n-3}}}
      \ncline{<->} {d}{e}  \naput{\sd_{n-6}}
      \ncline[offset=2pt]{<->} {f}{e}  \naput{\sd_{n-4}}
      \ncline[offset=-2pt]{<->} {f}{e} \nbput{\sd_{n-3}}
      \ncline{<->} {g}{f}  \naput{\sd_{n-2}}
      \ncline{<->} {h}{g}  \naput{\sd_{n-1}}
    \end{displaymath}
    \caption{\label{fig:nine}Relative positions for $n-7,\ldots,n+1$
      in $T^{\prime}$ (top left) and $X^{\prime}$ (bottom left).}
  \end{figure}

  Set $U^{\prime} = \sd_{n-1}(T^{\prime})$. Since the shape of
  $U^{\prime}$ restricted to entries up to $n$ is $\varepsilon$, we
  have $U^{\prime} \in [U]_{(2,n-2)}$ where $U = \sd_{n-1}(T)$,
  and $V^{\prime} \in [V]_{(2,n-2)}$ where $V = \sd_{n-2}
  \sd_{n-4} \sd_{n-6} \sd_{n-4} \sd_{n-2}
  (U^{\prime})$; again, see Figure~\ref{fig:nine}. Set $X^{\prime} =
  \sd_{n-1}(V^{\prime})$. Since the shape of $X^{\prime}$
  restricted to entries up to $n$ is $\zeta$, we have $X^{\prime} \in
  [X]_{(2,n-2)}$ where $X = \sd_{n-1}(V)$. Moreover, since
  $X^{\prime} = \sd_{n-1} \sd_{n-2} \sd_{n-4}
  \sd_{n-6} \sd_{n-4} \sd_{n-2} \sd_{n-1}
  (T^{\prime})$, we have that $X^{\prime} \in
         [T^{\prime}]_{(n-6,n-1)}$ as desired.
\end{proof}

Theorem~\ref{thm:positivity_sh} now follows as a corollary to
Theorem~\ref{thm:isomorphic_sh} and \eqref{e:classes_sh}.

%%%%%%%%%%%%%%%%%%%%%%%%%%%%%%%%%%%%%%%%%%%%%%%%%%%%%%%%%%%%%%%%
\section{Products of Schur $P$-functions}
%%%%%%%%%%%%%%%%%%%%%%%%%%%%%%%%%%%%%%%%%%%%%%%%%%%%%%%%%%%%%%%%
\label{s:weak}

We now present a first application of shifted dual equivalence. For
$\gamma \subseteq \varepsilon$ strict partitions, we define the
shifted skew diagram $\varepsilon/\gamma$ to be the set theoretic
difference between $\varepsilon$ and $\gamma$. For example, the
shifted skew diagram for $(6,4,3,1)/(5,2)$ is given in
Figure~\ref{fig:skew}.

\begin{figure}[ht]
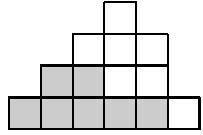

  \begin{center}
    \begin{displaymath}
      \tableau{ &  &  & \e \\  &  & \e & \e & \e \\ & \cb & \cb & \e & \e \\ \cb & \cb & \cb & \cb & \cb & \e}
    \end{displaymath}
    \caption{\label{fig:skew}The shifted Young diagram and shifted symmetric diagram for $(6,4,3,1)$.}
  \end{center}
\end{figure}

The combinatorial definitions for Schur $Q$-functions extend to skew
shifted diagrams \cite{Macdonald}, and the quasisymmetric expansions
in Section~\ref{s:quasi} hold for the shifted case as well. Precisely,
we have
\begin{equation}
  Q_{\varepsilon/\gamma} (X) \ = \ \sum_{S \in \shSYT(\varepsilon/\gamma)}
  2^{|\Peak(S)|+1} G_{\Peak(S)}(X)
\label{e:skew_Q}
\end{equation}
where $\shSYT(\varepsilon/\gamma)$ denotes the set of all standard
shifted tableaux of skew shifted shape $\varepsilon/\gamma$.

Schur $Q$ and $P$-functions have the same relationship as before,
though one must take care to track how many cells now live on the main
diagonal. The relation in \eqref{e:schur_P} becomes
\begin{equation}
  P_{\varepsilon/\gamma} (X) \ = \ 2^{\ell(\varepsilon) -
    \ell(\gamma)} Q_{\varepsilon/\gamma} (X) . 
\label{e:skew_P}
\end{equation}

Recall that the product of two Schur $P$-functions may be expanded in
the Schur $P$-function basis, so we may define integers
$f_{\gamma,\delta}^{\varepsilon}$ by
\begin{equation}
  P_{\gamma}(X) P_{\delta}(X) = \sum_{\varepsilon}
  f_{\gamma,\delta}^{\varepsilon} P_{\varepsilon}(X).
  \label{e:struct_P}
\end{equation}
Since the Schur $Q$- and $P$-functions form dual bases and the
operation of skewing is adjoint to multiplication
\cite{Macdonald}, these integers may also be defined by
\begin{equation}
  Q_{\varepsilon/\gamma}(X) = \sum_{\delta}
  f_{\gamma,\delta}^{\varepsilon} Q_{\delta}(X). 
  \label{e:struct_Q}
\end{equation}

Using the machinery of shifted dual equivalence, the second half of
Theorem~\ref{thm:P_pos} now follows.

\begin{corollary}
  For $\gamma \subseteq \varepsilon$ strict partitions,
  $f_{\gamma,\delta}^{\varepsilon}$ is the number of shifted dual
  equivalence classes of standard shifted tableaux of skew shifted
  shape $\varepsilon/\gamma$ under the action of
  $\{\sd_i\}_{1<i<|\varepsilon|-|\gamma|-1}$ that are isomorphic to
  $\shSYT(\delta)$. In particular, $f_{\gamma,\delta}^{\varepsilon}$
  is a nonegative integer.
  \label{cor:LRR}
\end{corollary}

%%%%%%%%%%%%%%%%%%%%%%%%%%%%%%%%%%%%%%%%%%%%%%%%%%%%%%%%%%%%%%%%
%
%  Bibliography
%
%%%%%%%%%%%%%%%%%%%%%%%%%%%%%%%%%%%%%%%%%%%%%%%%%%%%%%%%%%%%%%%%

\bibliographystyle{alpha} 
\bibliography{shifted.bib}

\begin{thebibliography}{BHRY}

\bibitem[AB12]{AssafBilley2012}
Sami~H. Assaf and Sara~C. Billey.
\newblock Affine dual equivalence and {$k$}-{S}chur functions.
\newblock {\em J. Comb.}, 3(3):343--399, 2012.

\bibitem[ABS]{ABS2014}
Sami Assaf, Nantel Bergeron, and Frank Sottile.
\newblock A combinatorial proof that {S}chubert vs. {S}chur coefficients are
  nonnegative.

\bibitem[Ass]{Assaf-DEGs-II}
Sami~H. Assaf.
\newblock Dual equivalence graphs {II}: {A} combinatorial proof of {LLT} and
  {M}acdonald positivity.
\newblock preprint, 2015.

\bibitem[Ass15]{Assaf-DEGs-I}
Sami~H. Assaf.
\newblock Dual equivalence graphs {I}: a new paradigm for {S}chur positivity.
\newblock {\em Forum Math. Sigma}, 3:e12, 33, 2015.

\bibitem[BHRY]{BilleyEtAl}
Sara Billey, Zach Hamaker, Austin Roberts, and Ben Young.
\newblock {C}oxeter-{K}nuth graphs and a signed {L}ittle map for type {$B$}
  reduced words.

\bibitem[Cho13]{Cho2013}
Soojin Cho.
\newblock A new {L}ittlewood-{R}ichardson rule for {S}chur {$P$}-functions.
\newblock {\em Trans. Amer. Math. Soc.}, 365(2):939--972, 2013.

\bibitem[Ges84]{Gessel1984}
Ira~M. Gessel.
\newblock Multipartite {$P$}-partitions and inner products of skew {S}chur
  functions.
\newblock In {\em Combinatorics and algebra (Boulder, Colo., 1983)}, volume~34
  of {\em Contemp. Math.}, pages 289--317. Amer. Math. Soc., Providence, RI,
  1984.

\bibitem[Hai92]{Haiman1992}
Mark~D. Haiman.
\newblock Dual equivalence with applications, including a conjecture of
  {P}roctor.
\newblock {\em Discrete Math.}, 99(1-3):79--113, 1992.

\bibitem[J{\'o}z91]{Jozefiak1991}
Tadeusz J{\'o}zefiak.
\newblock Schur {$Q$}-functions and cohomology of isotropic {G}rassmannians.
\newblock {\em Math. Proc. Cambridge Philos. Soc.}, 109(3):471--478, 1991.

\bibitem[Mac95]{Macdonald}
I.~G. Macdonald.
\newblock {\em Symmetric functions and {H}all polynomials}.
\newblock Oxford Mathematical Monographs. The Clarendon Press Oxford University
  Press, New York, second edition, 1995.
\newblock With contributions by A. Zelevinsky, Oxford Science Publications.

\bibitem[Pra91]{Pragacz1991}
Piotr Pragacz.
\newblock Algebro-geometric applications of {S}chur {$S$}- and
  {$Q$}-polynomials.
\newblock In {\em Topics in invariant theory ({P}aris, 1989/1990)}, volume 1478
  of {\em Lecture Notes in Math.}, pages 130--191. Springer, Berlin, 1991.

\bibitem[Sag87]{Sagan1987}
Bruce~E. Sagan.
\newblock Shifted tableaux, {S}chur {$Q$}-functions, and a conjecture of {R}.
  {S}tanley.
\newblock {\em J. Combin. Theory Ser. A}, 45(1):62--103, 1987.

\bibitem[Sch11]{Schur1911}
I.~Schur.
\newblock \"{U}ber {G}ruppen linearer {S}ubstitutionen mit {K}oeffizienten aus
  einem algebraischen {Z}ahlk\"orper.
\newblock {\em Math. Ann.}, 71(3):355--367, 1911.

\bibitem[Ser10]{Serrano2010}
Luis Serrano.
\newblock The shifted plactic monoid.
\newblock {\em Math. Z.}, 266(2):363--392, 2010.

\bibitem[Ste89]{Stembridge1989}
John~R. Stembridge.
\newblock Shifted tableaux and the projective representations of symmetric
  groups.
\newblock {\em Adv. Math.}, 74(1):87--134, 1989.

\bibitem[Wor84]{Worley1984}
Dale~Raymond Worley.
\newblock {\em A {T}HEORY OF {S}HIFTED {Y}OUNG {T}ABLEAUX}.
\newblock ProQuest LLC, Ann Arbor, MI, 1984.
\newblock Thesis (Ph.D.)--Massachusetts Institute of Technology.

\end{thebibliography}

\end{document}